\documentclass[a4paper,10pt]{amsart}

\usepackage[colorlinks,urlcolor=black,citecolor=black,linkcolor=black]{hyperref}
\usepackage{graphicx}
\usepackage{amsthm}
\usepackage{amssymb}
\usepackage{amsmath}
\usepackage{enumerate}

\newcommand{\nint}{\int_{-\infty}^{+\infty}}
\newcommand{\R}{\mathbb{R}}
\newcommand{\G}{G}
\newcommand{\RN}{\mathbb{R}}
\newcommand{\C}{\mathbb{C}}
\newcommand{\var}{\varepsilon}
\newcommand{\dist}{\mathrm{dist}}
\newcommand{\E}{{E}}
\newcommand{\Es}{E_S}
\newcommand{\Er}{E^*}
\newcommand{\Ers}{E_r ^*}
\newcommand{\I}{I}

\newcommand{\Ir}{I^*}

\newcommand{\Ch}{C}
\newcommand{\Chr}{C^*}
\newcommand{\Chrr}{C^*_r}

\newcommand{\PH}{P}
\newcommand{\X}{\mathbf{X}}

\newcommand{\M}{M}
\newcommand{\Mr}{M^*}
\newcommand{\V}{W}
\newcommand{\Gw}{\Gamma}
\newcommand{\GSc}{\Gamma_\sigma}
\newcommand{\wk}{\rightharpoonup}

\renewcommand{\Re}{\mathrm{Re}}
\renewcommand{\Im}{\mathrm{Im}}

\theoremstyle{definition}
\newtheorem{cor}{Corollary}[]
\newtheorem{lemma}{Lemma}[]
\newtheorem{remark}{Remark}[section]
\newtheorem{theorem}{Theorem}[]

\newtheorem{proposition}{Proposition}

\newtheorem*{definition*}{Definition}

\title[]{Multiple normalized standing-waves
solutions to the scalar non-linear Klein-Gordon equation with two competing powers}

\author[Daniele Garrisi]{}
\email{daniele.garrisi@nottingham.edu.cn}
\thanks{The author was supported by the School of Mathematical Sciences of the
University of Nottingham Ningbo China.}
\begin{document}
\begin{abstract}
In this work we prove the existence of standing-wave solutions to the scalar
non-linear Klein-Gordon equation in dimension one and the stability of the ground-state, the
set which contains all the minima of the energy constrained to the manifold of the states
sharing a fixed charge. For non-linearities which are combinations of two competing powers
we prove that standing-waves in the ground-state are orbitally stable. We also show the 
existence of a degenerate minimum and the existence of two positive and radially symmetric
minima having the same charge.
\end{abstract}
\maketitle
\thispagestyle{empty}
\centerline{\scshape Daniele Garrisi}
\medskip
{\footnotesize
\centerline{Room 324, Sir Peter Mansfield Building}
\centerline{School of Mathematical Sciences}
\centerline{University of Nottingham Ningbo China}
\centerline{199 Taikang East Road}
\centerline{315100, Ningbo, China}}
\section{Introduction}
In this work we address the problem of existence and stability of standing-wave solutions
to a non-linear Klein-Gordon equation
\begin{equation}
\label{eq.NLKG}
(\partial_{tt}^2 - \partial_{xx}^2 + m^2)\phi + \G'(|\phi|)\cdot\frac{\phi}{|\phi|} = 0
\end{equation}
where $ \G $ is a real-valued even function defined on $ (-\infty,+\infty) $
such that $ \G'(0) = 0 $. A standing-wave is a solution to \eqref{eq.NLKG} which can be 
written as
\begin{equation}
\label{eq.sw}
\phi(t,x) := e^{-i\omega t} R(x),\quad (t,x)\in\R\times\R
\end{equation}
where $ \omega $ is a real number and $ R $ is a real-valued function of class
$ H^1 (\R;\R) $. The standing-waves we are interested on satisfy the following variational characterization: they are minima of the energy functional $ \E $ on the constraint $ \M_\sigma $ which depends on a real parameter $ \sigma $. We set $ \X := H^1 (\mathbb{R};\C)\times L^2 (\mathbb{R};\C) $. For the vectors of this space we will use the notation $ \Phi $, and $ (\phi,\phi_t) $ for its components. The energy and constraint functional are complex-valued functions defined on $ \X $. Given $ (\phi,\phi_t) $ in $ \X $, we define
\begin{gather}
\label{eq.1}
\E(\phi,\phi_t) := \frac{1}{2}\nint |\phi_t(x)|^2 dx + 
\frac{1}{2}\nint|\phi'(x)|^2 dx + \nint\V(\phi(x))dx\\
\Ch(\phi,\phi_t) := -\Im\nint \phi_t(x)\overline{\phi(x)} dx,
\end{gather}
where
\begin{equation}
\label{eq.47}
\V(z) := \frac{1}{2} m^2 |z|^2 + \G(|z|).
\end{equation}
In fact, $ \E $ is a real-valued functional. We will refer to $ \Ch $ with the term \textsl{charge}. The constraint set is defined as
\[
\M_\sigma := \{(\phi,\phi_t)\in\X\mid\Ch(\phi,\phi_t) = \sigma\}\subseteq\X.
\]
Given two vectors $ \Phi = (\phi,\phi_t) $ and $ \Psi = (\psi,\psi_t) $ in $ \X $, we consider the scalar product
\[
(\Phi,\Psi)_\X := \text{Re}\nint\phi\overline{\psi}dx + \text{Re}\nint\phi_t\overline{\psi_t}dx.
\]
On $ \X $ we consider the metric $ d $ induced by the scalar product. In order to define stable subsets of $ \X $, we assume that $ \G $ is such that \eqref{eq.NLKG} is globally well-posed as meant in \cite[Remark~3.5,~p.~126]{Tao06}. That is, given an initial datum $ (\phi_0,\phi_{t,0}) $ in $ H^1\times L^2 $, there exist a unique solution $ \phi $ defined on 
$ (-\infty,+\infty)\times\mathbb{R} $ such that
\begin{equation}
\label{eq.126}
\phi\in C_t^0 H\sp 1 _x ((-\infty,+\infty)\times\RN)\cap 
C\sp 1 _t L\sp 2 _x ((-\infty,+\infty)\times\RN).
\end{equation}
Moreover, $ \E $ and $ \Ch $ are constant on the trajectory
$ (\phi(t,\cdot),\partial_t\phi(t,\cdot)) $, that is
\begin{equation}
\label{eq.101}
\E(\phi(t,\cdot),\partial_t\phi(t,\cdot)) = 
\E(\phi_0,\phi_{t,0}),\quad 
\Ch(\phi(t,\cdot),\partial_t\phi(t,\cdot)) = 
\Ch(\phi_0,\phi_{t,0})
\end{equation}
for every $ t $ in $ (-\infty,+\infty) $, \cite[Remark~3.5,~p.~126]{Tao06}. Then for
every $ t $ in $ (-\infty,+\infty) $ we can define
\[
U(t,\cdot)\colon\X\to\X,\quad
U(t,(\phi_0,\phi_{t,0})) = (\phi(t,\cdot),\partial_t\phi(t,\cdot)).
\]
\begin{definition*}[Stable subsets of $ \X $]
A set $ S\subseteq\X $ is stable if, for every $ \var > 0 $, there exist 
$ \delta > 0 $ such that
\[
\dist(\Phi,S) < \delta\Rightarrow\dist(U(t,\Phi),S) < \var
\]
for every $ t $ in $ (-\infty,+\infty) $ and $ \Phi\in\X $.
\end{definition*}
In this work we present two stability results. The first is the stability of the subset of
$ \X $ called \textsl{ground state}, which is the set of all the minima of $ \E $ on the constraint $ \M_\sigma $. We use the notation
\[
\GSc := \{(\phi,\phi_t)\in\M_\sigma\mid\E(\phi,\phi_t) = \inf_{\M_\sigma} (\E)\}.
\]
The literature is rich of stability results to a variety of differential equations and
coupled differential equations, where the ground state is defined according to the
energy functional and constraint. Since similar techniques are also used, it is covenient
to mention results of stability in differential equations having different structure than
\eqref{eq.NLKG}. Some references are \cite{BBGM07,Shi14} for the non-linear Schr\"odinger equation (NLS) in dimension $ n\geq 3 $, or \cite{CL82} on the stability of (NLS) in dimension $ n\geq 1 $,
\cite{BBBM10} on the non-linear Klein-Gordon equation (NLKG) in dimension $ n\geq 3 $, 
\cite{Gar12} on the coupled non-linear Klein-Gordon (2-NLKG) equation in dimension $ n\geq 3 $,
\cite{GJ16,NW11,Iko14} and \cite{Oht96} on the coupled non-linear Schr\"odinger equation (2-NLS) in dimension $ n = 1 $. We include \cite{NTV14,NTV15,NTV19} which address problems of stability for the non-linear Schr\"odinger equation in bounded domains. For the stability of the ground state we require the following assumptions:
\begin{equation}
\tag{G1}
\label{R1}
\mu := \inf_{s\in(0,+\infty)}\frac{\V(s)}{s^2} > 0
\end{equation}
there exist $ s_0 > 0 $ such that
\begin{equation}
\tag{G2} 
\label{R2}
\G(s_0) < 0
\end{equation}
there exist $ 2 < p < q $ such that 
\begin{equation}
\tag{G3}
\label{R3}
|\G'(s)|\leq C(|s|^{p - 1} + |s|^{q - 1}),\quad \G(0) = 0.
\end{equation}
\hypertarget{R4}{(G4)}. The equation \eqref{eq.NLKG} is globally well-posed in 
$ H^1 (\R;\C)\times L^2 (\R;\C) $.
\begin{theorem}
\label{thm.ground-state-stability}
If (\ref{R1}-\hyperlink{R4}{G4}) 
hold, there exist $ \sigma_* > 0 $ such that $ \GSc $ is stable for every $ \sigma > \sigma_* $.
\end{theorem}
We prove the Concentration-Compactness property of the 
minimizing sequences of $ \E $ on $ \M_\sigma $. That is, given a minimizing sequence
$ (\phi_n,\phi_{n,t}) $, there exist $ (y_n)\subseteq\R $ such that a subsequence of 
$ (\phi_n(\cdot + y_n),\phi_{n,t} (\cdot + y_n)) $ converges strongly in $ \X $.
In previous references as \cite{BBBM10,Gar12}, the Concentration-Compactness had been proved 
for particular minimizing sequences, which are the ones satisfying 
$ \phi_{n,t} = -i\omega_n\phi_n $ for some
$ \omega_n $ in $ \R $. Equivalently, that implies the Concentration-Compactness property
of the minimizing sequences of the functional $ \Er $ and constraint $ \Mr_\sigma $ defined as follows:
\begin{gather}
\Er\colon H^1 (\R;\R)\times\R\to\R,\quad \Chr\colon H^1 (\R;\R)\times\R\to\R\\
\label{eq.4}
\Er(u,\omega) := \E(u,-i\omega u) = 
\frac{1}{2}\nint 
\Big(|u'(x)|\sp 2 + \omega^2 |u(x)|^2 + 2\V(u(x))\Big)dx\\
\label{eq.5}
\Chr(u,\omega) = \Ch(u,-i\omega u) = \omega\nint |u(x)|^2 dx
\end{gather}
and
\[
\Mr_\sigma := \{(u,\omega)\in H^1 (\R;\R)\times\R\mid\Chr(u,\omega) = \sigma\}.
\]
We address specifically the dimension $ n = 1 $, a case which was not covered in \cite{BBBM10,Gar12}. This presents different challenges. In fact, when $ n\geq 3 $, the functional $ \Er $ is coercive, provided $ q\leq\frac{2n}{n - 2} $ in \eqref{R3} and 
$ \V\geq 0 $. As we will show in Remark~\ref{rem.coercive}, in dimension $ n = 1 $, the coercivity fails if $ \V\geq 0 $ and $ \V $ has a zero different than the origin, no matter what restriction one sets on $ q $.
The second stability result is about the subset of $ \X $ which is obtained by taking all
the argument translations and multiplication by complex numbers in the unit sphere of an element of the ground state: given $ \Phi $ in $ \GSc $, we set
\begin{equation}
\label{eq.ssw}
\Gw(\Phi) := \{z\Phi(\cdot + y)\mid (z,y)\in S^1\times\R\}.
\end{equation}
The invariances
\begin{equation}
\label{eq.77}
\E(z\Phi(\cdot + y)) = \E(\Phi),\quad\Ch(z\Phi(\cdot + y)) = \Ch(\Phi)
\end{equation}
for every $ (z,y) $ in $ S^1\times\R $ and $ \Phi $ in $ \X $ show that $ \Gw(\Phi) $
is a subset of $ \GSc $. This result addresses specifically the double power non-linearity
\begin{equation}
\label{eq.DP}
G(s) = -as^4 + bs^6,\quad a,b > 0.
\end{equation}
\begin{definition*}[Orbitally stable standing wave]
The standing-wave in \eqref{eq.sw} is orbitally stable if $ \Gw(R,-i\omega R) $
is a stable subset of $ \X $.
\end{definition*}
References on the orbital stability of  $ \Gamma_\sigma $ do not always address
the set \eqref{eq.ssw}. 
In \cite{BBGM07,BBBM10,Gar12} and \cite{GJ16}, only the stability of the ground state has
been proved. Other references who proved the stability of the standing-wave took advantage of more specific assumptions. For instance, 
in \cite{CL82} the non-linearity $ \G $ is a pure-power 
\begin{equation}
\label{eq.PP}
G(s) = -a|s|^p,\quad a > 0,\quad 2 < p < 2 + \frac{4}{n}.
\end{equation}
In \cite{Oht96,NW11,LNW16,Cor16b} the non-linearity $ \G $ is an homogeneous two or more variables function or fourth degree homogeneous polynomial. This choice allows to prove that 
$ \Gw(\Phi) = \GSc $ for every $ \Phi $ element of the ground-state.
Therefore, the stability of $ \Gw(\Phi) $ follows from the stability of the ground-state, that is from Concentration-Compactness property of minimizing sequences which follows
by \cite{Lio84a,Lio84b}.
Such equality is consequence of the uniqueness of positive, decaying and symmetrically decreasing solutions to the elliptic problem $ \Delta R(x) - G'(R(x)) - R(x) = 0 $, \cite{MS87,Kwo89,McL93},
and the symmetry $ G(ts) = t^p G(s)\ (t > 0) $. In fact, one can show that
each set $ \Gw(\Phi) $ contains a unique minimum $ R $ which is positive, even, radially
decreasing and $ H^1 (\R;\R) $. This follows, for instance, from the conclusions of 
\cite{BJM09}. We included a proof in Theorem~\ref{thm.minima-structure}.
Therefore, it is convenient to define
\begin{equation}
\label{eq.critical-set}
K_\sigma := \GSc\cap H^1 _{r,+} (\R;\R),
\end{equation}
where $ H^1 _{r,+} (\R;\R) $ is the set of $ H^1 $ positive and even functions.
The choice of the powers 4 and 6 is in part motivated by historical reasons: in dimension 
$ n = 3 $, the non-linear Klein-Gordon with two competing powers \eqref{eq.DP} was proposed in \cite{Lee81} as a model of 0-spin particles in spin theory. In
\cite{GSS87,GSS90,Sha83,SS85} they proved existence of stable and unstable standing-waves, still in
dimension $ n = 3 $. Moreover, in \eqref{eq.DP} we are able to evaluate explicitly
the energy $ \E $ and the charge $ \Ch $ of the standing-wave \eqref{eq.1}, to prove the
stability of $ \Gw(R,-i\omega R) $ and to count the number of minima and detect which ones are non-degenerate. In order to describe these results, we introduce the notation:
\begin{equation}
\label{eq.127}
\tau(a,b,m) := \frac{2m^2b}{a^2}.
\end{equation}
The behaviour of the number of sets $ \Gw(\Phi) $ and the non-degeneracy of minima depend 
jointly on the non-linearity \eqref{eq.DP} and $ m $ (specifically on $ \tau(a,b,m) $), and
$ \sigma $. We list some of the main results: there exist $ \tau_* > 1 $ such that
\begin{enumerate}[(i)]
\itemsep=0.3em
\item If $ \tau\geq\tau_* $, then $ |K_\sigma| = 1 $ for every $ \sigma > 0 $. Therefore,
in every constraint there is exactly one positive, symmetric minimum from  Theorem~\ref{thm.multiplicity-result}. Therefore, 
$ \GSc = \Gw(\Phi) $ for every
$ \Phi $ in $ \GSc $ just as it happens for pure-powers or double powers in the non-linear
Schr\"odinger equation in \cite{CL82,GG17}
\item if $ \tau = \tau_* $ there exist exactly one level $ \sigma_d $ where the minimum
is degenerate, from Theorem~\ref{thm.degeneracy}. This is a completely different behaviour
than (NLS), where for double powers minima are non-degenerate, \cite{GG17}. Yet the corresponding standing-wave is stable
\item if $ 1 < \tau < \tau_* $, there exist exactly one level $ \sigma_2 $ where
we can observe the existence of \textsl{two} positive, symmetrically decreasing minima,
that is $ |K_\sigma| = 2 $, Theorem~\ref{thm.multiplicity-result}. Consequently we have the disjoint union $ \GSc = \Gw(\Phi_1)\cup\Gw(\Phi_2) $. This is the most unexpected result. In
\cite{Bon10} the author proved that if $ \sigma $ is large enough,
one can expect to find at least as many local minima as the number of connected components of the set $ \{G < 0\} $. According Theorem~\ref{thm.multiplicity-result}, there are more local minima (in fact, minima) than the number of connected components. In fact for $ \G $ in \eqref{eq.DP}, the set is connected. In (2-NLS) one can find more multiplicity results, as \cite{BS19}, even if
critical points are not minima
\item the set $ \Gw(\Phi) $ is stable for every choice of $ \tau > 1 $ and $ \sigma $, 
Theorem~\ref{thm.stability-wave}. This applies in particular to the case where there is
a unique set $ \Gw(\Phi) $, but also in the case where there are two minima and, consequently,
there are two different sets $ \Gamma(\Phi_1) $ and $ \Gamma(\Phi_2) $ disjoint from each other. In \cite{Gar12} we already pointed out methods to prove the stability of $ \Gamma(\Phi) $
regardless of the cardinality of $ K_\sigma $. This is a concrete example showing that the uniqueness of $ \Gamma(\Phi) $ is not necessary.
\end{enumerate}
The paper is organized as follows. In \S\ref{sect.cc} we prove properties of the pairs
$ (\E,\M_\sigma) $ and $ (\Er,M_\sigma^*) $ and the stability of the ground-state. In \S\ref{sect.du}, we count the number of positive and
radially decreasing minima in every constraint $ \M_\sigma $, the cardinality of $ K_\sigma $. 
The result is summarized and
proved in Theorem~\ref{thm.multiplicity-result}. In \S\ref{sect.dd}, we look at the 
degeneracy of minima through Theorem~\ref{thm.degeneracy}. The paper concludes with Theorem~\ref{thm.stability-wave} of \S\ref{sect.sw}, where we prove that all the standing-waves in \eqref{eq.1} such that $ (R,-i\omega R) $ is in $ \GSc $ are orbitally stable.
\smallskip

Throughout all the work it is assumed that $ \sigma > 0 $. Theorems~\ref{thm.multiplicity-result},
\ref{thm.degeneracy} and \ref{thm.stability-wave} can be proved in the case $ \sigma < 0 $ by
observing that the isometry from $ \M_\sigma $ to $ \M_{-\sigma} $ which maps $ (\phi,\phi_t) $
to $ (\phi,-\phi_t) $, does not change the energy $ \E $. 
The case $ \sigma = 0 $ has been addressed separately at the end of \S\ref{sect.sw}.
\section{Properties of the functional $ \E $}
\label{sect.cc}
Some of the properties we are going to prove have a correspondence
with the variational setting of the non-linear Schr\"odinger equation,
treated in \cite{BBGM07}. Therefore, it is convenient to introduce
a notation for the functional
\begin{equation}
\label{eq.49}
\E_S(\phi) := \frac{1}{2} \nint |\phi'(x)|\sp 2 + \nint \G(|\phi(x)|)dx
\end{equation}
and the constraint
\[
S(\lambda) := \{u\in H^1(\R;\C)\mid\|u\|_{L^2}^2 = \lambda\}.
\]
The functionals $ \E $ and $ \Er $, defined in \eqref{eq.1} and \eqref{eq.4}
can be related to each other through the following map
\begin{equation}
\label{eq.11}
P\colon\X\to H^1 (\R;\R)\times\R,\quad
P(\phi,\phi_t) := \left(|\phi|,\frac{\sigma}{\|\phi\|_{L^2}^2}\right)
\end{equation}
which is the one given in \cite[(3.19)]{BBBM10}. Under the assumptions (\ref{R1}-\ref{R3}),
we can prove the following proposition.
\begin{proposition}
\label{prop.1}
The functional $ \E $ satisfies the following properties:
\begin{enumerate}[(i)]
\item\label{prop.1.1} 
$ \Er(P(\Phi))\leq \E(\Phi) $ and $ \Chr(P(\Phi)) = \Ch(\Phi) $ 
for every $ \Phi $ in $ \X $
\item\label{prop.1.2}
$ \E\geq 0 $ on $ \M_\sigma $. Then it is bounded from below on $ \M_\sigma $
\item\label{prop.1.3} 
there exist a continuous real valued function $ h_1 $ such that $ h_1 (0) = 0 $ and
for every $ e $ in $ \R $
\[
\E(\Phi)\leq e\implies\|\Phi\|_X\leq h_1 (e)
\]
for every $ \Phi $ in $ \M_\sigma $ and
\[
\Er(u,\omega)\leq e\implies\|u\|_{H^1}\leq h_1(e)
\]
for every $ (u,\omega) $ in $ \M_\sigma^* $
\item\label{prop.1.4}
there exist a continuous
function $ h_2 $ from $ \mathbb{R}\times\mathbb{R}\times\X $ to $ \R $ which is
bounded on bounded subsets and such that
\[
|\E(\lambda\Phi) - \E(\mu\Phi)|\leq h_2(\lambda,\mu,\Phi)
|\lambda - \mu|
\]
for every triple $ (\lambda,\mu,\Phi) $ in $ \R\times\R\times\X $ 
\item\label{prop.1.5}
$ \E $ and $ \Ch $ are continuously differentiable.
\end{enumerate}
\end{proposition}
\begin{proof}
\eqref{prop.1.1}.
This property has been already thoroughly proved in \cite[Lemma~3.3]{BBBM10}.
Here we add some intermediate inequalities. We set
\[
\omega := \frac{\Ch(\Phi)}{\|\phi\|_{L^2} ^2}.
\]
Then, from \eqref{eq.11}, $ P(\Phi) = (|\phi|,\omega) $.
\begin{equation}
\label{eq.10}
\begin{split}
\E(\phi,\phi_t) &= \frac{1}{2}\nint |\phi_t(x)|\sp 2 dx +
\frac{1}{2}\nint \Big(|\phi'(x)|\sp 2 + 2\V(\phi(x))\Big)dx\\
&\geq\frac{1}{2} \nint \Big(||\phi|'(x)|\sp 2 + 2\V (\phi(x))\Big)dx
+ \frac{1}{2}\cdot\frac{|\Ch(\Phi)|^2}{\|\phi\|_{L\sp 2}\sp 2} \\
&= \frac{1}{2} \nint \Big(||\phi|'(x)|\sp 2 + 2\V (|\phi(x)|)\Big)dx
+ \frac{1}{2}\omega^2 \|\phi\|_{L^2} ^2 \\
&= \Er(|\phi|,\omega) = \Er(\PH(\Phi)).
\end{split}
\end{equation}
The second inequality follows from \cite[Theorem~7.8,~p.,~177]{LL01}, 
the Convex Inequality for the gradient, and the Cauchy-Schwarz
inequality applied to the scalar product 
\[
(f,g)_{L^2} := \Re\nint f(x)\overline{g}(x)dx,\quad f = i\phi,\quad g = \phi_t.
\]
\eqref{prop.1.2}. From \eqref{R1}, it follows that $ \V\geq 0 $. Then $ \E(\Phi)\geq 0 $ for
every $ \Phi $ in $ \X $.

\eqref{prop.1.3}. We set $ \Phi = (\phi,\phi_t) $. Again, from \eqref{R1}, we have
\begin{equation*}
2e\geq 2\E\geq \|\phi_t\|_{L^2} ^2 + \|\phi'\|_{{L^2}^2} + \mu\|\phi\|_{{L^2}^2}
\geq\min\{\mu,1\}\|\Phi\|_{{\X}^2}.
\end{equation*}
Then we can define
\begin{equation*}
h_1 (e) := \left(\frac{2e}{\min\{\mu,1\}}\right)^{\frac{1}{2}}.
\end{equation*}
Given $ (u,\omega) $ in $ M_\sigma^* $ if $ \Er(u,\omega)\leq e $, then 
$ \E(u,-i\omega u)\leq e $ by \eqref{eq.4}. Then
\[
\|u\|_{H^1}\leq\|(u,-i\omega u)\|_X\leq h_1(e).
\]
\eqref{prop.1.4}. We can write
\begin{equation}
\label{eq.50}
\begin{split}
|\E(\lambda\Phi) - \E(\mu\Phi)|&\leq
\frac{1}{2}|\lambda^2 - \mu^2|\nint|\phi_t (x)|^2 dx \\
&+ \frac{m^2}{2}|\lambda^2 - \mu^2|\nint|\phi(x)|^2dx \\
&+ \nint |\E_S(\lambda\phi(x)) - \E_S(\mu\phi(x))|dx.
\end{split}
\end{equation}
From (ii) of \cite[Proposition~1]{GG17}, there exist a real-valued function
$ c $ defined on the set $ \R\times\R\times H^1 (\R;\C) $ which is bounded on
bounded sets and such that
\begin{equation}
\label{eq.51}
|\E_S(\lambda\phi) - \E_S(\mu\phi)|\leq |\lambda - \mu|c(\lambda,\mu,\phi).
\end{equation}
In order to apply the quoted proposition, we only need to check that $ \G $
fulfills the property \cite[(G2a)]{GG17}, which corresponds to \eqref{R2}. From \eqref{eq.50} and
\eqref{eq.51}, we obtain
\[
\begin{split}
|\E(\lambda\Phi) - \E(\mu\Phi)|\leq&\frac{1}{2}|\lambda^2 - \mu^2|\|\phi_t\|_{{L^2}^2} + 
\frac{m^2}{2}|\lambda^2 - \mu^2|\|\phi\|_{{L^2}^2} + |\lambda - \mu|c(\lambda,\mu,\phi)\\
&\frac{1}{2}|\lambda - \mu|\left(|\lambda + \mu|\|\phi_t\|_{{L^2}^2} + m^2|\lambda + \mu|
\|\phi\|_{{L^2}^2} + 2c(\lambda,\mu,\phi)\right).
\end{split}
\]
Then, the conclusion follows if we define
\[
h_2 (\lambda,\mu,\Phi) = \frac{1}{2}(|\lambda + \mu|\|\phi_t\|_{{L^2}^2} + m^2|\lambda + 
\mu|\|\phi\|_{{L^2}^2} + 2c(\lambda,\mu,\phi)).
\]
\eqref{prop.1.5}. Here we refer to \cite[Proposition~6]{GG17}, where we showed that under the assumption \eqref{R2}, the functional $ \E_S $ defined in \eqref{eq.49} is
continuously differentiable. The quoted proposition uses the same techniques as \cite[Theorem~2.2,~p.~16]{AP93}. Therefore,
$ \E $ is the sum of $ \E_S $ and a half of the squares of the $ L^2 $-norm of $ \phi_t $ and 
$ \phi $, which are continuously differentiable functions. As for $ \Ch $, we have
\[
\Ch(\phi + h,\phi_t + k) = \Ch(\phi,\phi_t) -\text{Im}\nint k\overline{\phi}dx 
-\text{Im}\nint \phi_t\overline{h}dx - \text{Im}\nint k\overline{h}dx
\]
showing that $ \Ch $ is continuously differentiable as the last term is $ o(h,k) $.
\end{proof}
\begin{remark}
\label{rem.coercive}
In the non-linear Klein-Gordon in dimension $ n\geq 3 $, the authors of \cite{BBBM10} prove
the coercivity of $ \Er $ with the sub-critical assumption $ q < \frac{2n}{n - 2} $ and a weaker assumption than \eqref{R1}, namely $ \V\geq 0 $ and \eqref{R1} is satisfied in a neighbourhood
of the origin. In the proof the use the fact that the $ L^{\frac{2n}{n - 2}} $-norm is estimated from above by the $ L^2 $-norm of the gradient, something which does not hold in dimension one for any $ L^p $-norm. In fact, if \eqref{R1} does not hold, the coercivity fails if 
$ \V\geq 0 $ and $ \V $ has a zero different from the origin. Given $ \sigma > 0 $ and
$ (u,\omega) $ in $ M_\sigma^* $, there holds
\begin{equation*}
\Er(u,\omega) = \frac{1}{2}\|u'\|_{L^2}^2 + \frac{\sigma^2}{2\|u\|_{L^2}^2} + \nint \V(u(x))dx.
\end{equation*}
Suppose that $ \V\geq 0 $ and there exist $ s_0 > 0 $ such that $ \V(s_0) = 0 $, which clearly
contradicts \eqref{R1}. Then $ \Er $ is non-coercive on $ \M_\sigma^* $ for
every choice of $ \sigma $. In fact, for every integer $ k\geq 1 $ consider the test function
\[
u_k(x) := 
\left\{
\begin{array}{ll}
s_0 & \text{ if } 0 \leq |x|\leq k\\
s_0\left(k + 1 - |x|\right) & \text{ if } k\leq |x|\leq k + 1\\
0 & \text{ if } k + 1\leq |x|.
\end{array}
\right.
\]
The function $ u_k $ is a continuous piecewise linear non-decreasing function with compact
support. Therefore, $ u_k $ is $ H^1(\R;\R) $. We have
\[
\nint\V(u_k(x))dx = 2 \V(s_0) k + 2 \int_k^{k + 1} \V\left(s_0(k + 1 - x)\right)dx.
\]
The variable change
\[
s_0\left(k + 1 - x\right) = t
\]
and the assumption $ \V(s_0) = 0 $ yield
\begin{equation*}
\nint \V(u_k(x))dx = 2\V(s_0)k + \frac{2}{s_0}\int_0^{s_0} \V(t)dt = \frac{2}{s_0}\int_0^{s_0} \V(t)dt.
\end{equation*}
Using the same variable change, 
\begin{equation}
\label{eq.72}
\nint u_k^2dx = 2s_0^2k + \frac{2}{s_0}\int_0^{s_0} t^2 dt = 2s_0^2k + \frac{2s_0^2}{3} = 
2s_0^2\left(k + \frac{1}{3}\right)
\end{equation}
and
\begin{equation}
\label{eq.73}
\nint |u_k'|^2dx = 2s_0^2.
\end{equation}
Therefore, if we set $ \omega_k := \sigma \|u_k\|_{L^2}^{-2} $,
\[
\Er(\omega_k,u_k) = s_0^2 + \frac{3\sigma^2}{4s_0^2(3k + 1)} + \frac{2}{s_0}\int_0^{s_0} \V(t)dt
\]
Therefore $ ((u_k,\omega_k))_{k\geq 1} $ is an unbounded sequence in $ H^1 $, from \eqref{eq.72}.
However, $ (E(u_k,\omega_k))_{k\geq 1} $ is bounded, which shows that $ \Er $ is not
coercive on $ \M_\sigma^* $. If we set $ \Phi_k := (u_k,-i\omega_k u_k) $ the sequence of
$ \Phi_n $ is unbounded in $ \X $, while $ \E(\Phi_k) = \Er(u_k,\omega_k) $, from \eqref{eq.4}.
Therefore $ \E $ is not coercive bounded on $ \M_\sigma $. Using $ u_k $ as test function suggests one of the differences between the
case $ n = 1 $ and $ n\geq 2 $. In the latter, the Lebesgue measure of an annulus of fixed width 1 diverges as the radius $ k $ diverges as $ k\to\infty $, while it is constant in the former. Therefore, in \eqref{eq.73}
an integral power of $ k $ should have appeared when $ n\geq 2 $, which would have made the sequence $ (\Er(u_k,\omega_k))_{k\geq 1} $ diverge.
\end{remark}
From the combined power-type estimate \eqref{R3}, we can obtain a combined power-type estimate for $ \G $.
In fact,
\begin{equation}
\label{eq.53}
|\G(s)| \leq \int_0^s |\G'(t)|dt\leq C\int_0^s (|t|^{p - 1} + |t|^{q - 1})dt\leq
\frac{C}{p}|s|^p + \frac{C}{q}|s|^q.
\end{equation}
From \eqref{prop.1.2} of Proposition~\ref{prop.1} both the $ \E $ and $ \Er $ are bounded below
on $ \M_\sigma $ and $ \M_\sigma^* $, respectively. Therefore, the following notations
\[
\I(\sigma) := 
\inf\{\E(\phi,\phi_t)\mid (\phi,\phi_t)\in\M_\sigma\},\quad
\Ir(\sigma) := \inf\{\Er(u,\omega)\mid (u,\omega)\in \M_\sigma^*\}.
\]
are justified.
\begin{proposition}
\label{prop.2}
For every $ \sigma > 0 $, the function $ I $ satisfies the following properties:
\begin{enumerate}[(i)]
\item
\label{prop.2.1}
$ I(\sigma)\leq\sigma m $ 
\item
\label{prop.2.2}
for every $ \vartheta\geq 1 $ and $ \sigma > 0 $, there holds
$ I(\vartheta\sigma)\leq\vartheta I(\sigma) $. 
If the equality holds, then $ \vartheta = 1 $ or $ I(\sigma) = \sigma m $
\item
\label{prop.2.4}
$ \I\colon(0,+\infty)\to\R $ is a continuous function
\item
\label{prop.2.3}
if \eqref{R2} holds, there exist $ \sigma_* \geq 0 $ such that
\[
\frac{I(\sigma)}{\sigma} < m\text{ on } (\sigma_*,+\infty),
\quad \frac{I(\sigma)}{\sigma} = m\text{ on } (0,\sigma_*].
\]
If there exist $ \sigma $ such that $ 0 < \sigma < \sigma_* $, then
$ \GSc $ is empty.
\end{enumerate}
\end{proposition}
\begin{proof}
\eqref{prop.2.1}. 
Firstly, we show that 
\begin{equation}
\label{eq.132}
\I(\sigma) = \Ir(\sigma).
\end{equation}
From the definition of $ \Er $ and $ \Chr $ in \eqref{eq.4} and \eqref{eq.5}, it
follows that $ \inf_{\M_\sigma}(\E)\leq \inf_{M_\sigma^*}(\Er) $.
The converse inequality follows from \eqref{prop.1.1} of Proposition~\ref{prop.1}.
Now, given $ (u,\omega) $ in $ M_\sigma^* $, from \eqref{eq.4}, \eqref{eq.49} and the definition of
$ \V $ in \eqref{eq.47}, we have
\begin{equation}
\label{eq.41}
\Er(u,\omega) = \frac{1}{2}\left(\frac{\sigma^2}{\|u\|_{L^2}^2} + m^2\|u\|_{L^2}^2\right) + \E_S(u).
\end{equation}
because $ \omega\|u\|_{L^2}^2 = \sigma $. We set $ \lambda := \frac{\sigma}{m} $. Then, if 
$ \|u\|_{L^2} ^2 = \frac{\sigma}{m} $ and $ \omega = m $, we obtain
$ E(u,\omega) = \sigma m + \E_S(u) $. From (i) of \cite[Lemma~2.3]{Shi14} 
$ \inf_{S(\lambda)}(\E_S) \leq 0 $. Therefore $ \Ir(\sigma)\leq\sigma m $.

\eqref{prop.2.2}.
Let $ (u_n,\omega_n) $ be a minimizing sequence of $ \Er $ on 
$ M_{\sigma}^* $. From \eqref{prop.1.3} of Proposition~\ref{prop.1}, the sequence $ (u_n,\omega_n) $ is bounded. Then
up to extract a subsequence we can suppose that there exist $ L $ such that 
\begin{equation}
\label{eq.128}
\lim_{n\to\infty} \|u_n'\|_{L^2}^2 = L_1,\quad\sup_{n\geq 1} \|u_n\|_{L^2}^2\leq L_0
\end{equation}
We define 
\[
u_{n,\vartheta} (x) := u_n (\vartheta\sp{-1} x).
\]
From the variable change $ y = \vartheta\sp{-1} x $ it follows that
\begin{gather}
\label{eq.45}
\|u_{n,\vartheta}\|_{L^2}^2 = \vartheta \|u_n\|_{L^2}^2\\
\nint\V(u_{n,\vartheta} (x))dx = \vartheta \nint\V(u_n(x))dx\\
\|u_{n,\vartheta}'\|_{L^2}^2 = \vartheta^{-1} \|u_n'\|_{L^2}^2.
\end{gather}
Then
\begin{equation}
\label{eq.rescaling}
\begin{split}
\Er(u_{n,\vartheta},\omega_n) &= \frac{1}{2}\omega_n^2\vartheta\|u_n\|_{L^2}^2  
+ \frac{1}{2}\vartheta^{-1} \|u_n'\|_{L^2}^2 + \vartheta\nint\V(u_n(x)) dx\\
&= \vartheta \Er(u_n,\omega_n) - \frac{1}{2}\left(\vartheta - \vartheta^{-1}\right)
\|u_n'\|_{L^2}^2.
\end{split}
\end{equation}
From \eqref{eq.45} $ (u_{n,\vartheta},\omega_n) $ belongs to $ \M_{\vartheta\sigma}^* $. From \eqref{eq.rescaling} and
\eqref{eq.128}, it follows that
\begin{equation}
\label{eq.129}
I(\vartheta\sigma)\leq \vartheta I(\sigma) -
\frac{1}{2}(\vartheta - \vartheta^{-1})\|u_n'\|_{L^2}^2 + o(1)\leq \vartheta I(\sigma) + o(1)
\end{equation}
because $ ((u_n,\omega_n)) $ is a minimizing sequence, proving the inequality. If the equality
$ I(\vartheta\sigma) = \vartheta I(\sigma) $ holds then either $ \vartheta = 1 $ or $ L_1 = 0 $.
Then, from the Sobolev-Gagliardo-Nirenberg inequality in dimension one we have
\begin{equation}
\label{eq.122}
\|u_n\|_{L^r}^r\leq s_{GN}^r\|u_n'\|_{L^2}^{\frac{r - 2}{2}}\|u_n\|_{L^2}^{\frac{r + 2}{2}}
\end{equation}
for $ r = p,q $. From \eqref{eq.128} and $ L_1 = 0 $ the sequences of $ L^p $ and $ L^q $ norms
of $ u_n $ converge to zero. Then, from \eqref{R3} and the estimate \eqref{eq.53}, it follows that
\[
\lim_{n\to\infty}\nint\G(u_n(x))dx = 0.
\]
Then, $ \E_S(u_n)\to 0 $. From \eqref{eq.41}
\[
\Er(u_n,\omega_n) = \frac{1}{2}\left(\frac{\sigma^2}{\|u_n\|_{L^2}^2} + m^2\|u_n\|_{L^2}^2\right) + o(1)\geq \sigma m + o(1).
\]
which implies $ I(\sigma) = \sigma m $ by \eqref{prop.2.2}.

\eqref{prop.2.4}. We fix $ \sigma_0 > 0 $. Let $ 0 < \sigma $ be such that 
$ |\sigma_0 - \sigma|\leq 1 $ and $ (u_n,\omega_n) $ be a minimizing sequence 
in $ M_{\sigma}^* $ such that $ \Er(u_n,\omega_n)\leq\I(\sigma) + 1 $.
From \eqref{prop.2} $ \I(\sigma)\leq (\sigma + 1)m $ for every 
$ 0 < \sigma\leq\sigma_0 + 1 $. Therefore $ \Er(u_n,\omega_n)\leq (\sigma + 1)m $
for every $ n\geq 1 $. By \eqref{prop.1.3} of Proposition~\ref{prop.1}
\[
\|u_n'\|_{L^2}^2\leq h_1 ((\sigma + 1)m)\leq\sup_{e\in [0,\sigma_0 + 1]} h_1 (e) := M.
\]
We remark that $ M $ depends only on $ \sigma_0 $.
Hereafter, we will assume that minimizing sequences fulfill this estimate which is uniform
with respect to the level. Let $ h $ be such that $ |h|\leq 1 $ .
We apply \eqref{eq.129} with $ \vartheta = (\sigma_0 + h)/\sigma_0 $ and 
$ \sigma = \sigma_0 $:
\begin{equation*}
\I(\sigma_0 + h)\leq \frac{\sigma_0 + h}{\sigma_0}\I(\sigma_0) + \frac{1}{2}\left|
\frac{\sigma_0 + h}{\sigma_0} - \frac{\sigma_0}{\sigma_0 + h}\right| M.
\end{equation*}
Then,
\begin{equation}
\label{eq.130}
\I(\sigma_0 + h) - \I(\sigma_0)\leq \frac{h}{\sigma_0}\I(\sigma_0) + \frac{1}{2}
\left|\frac{\sigma_0 + h}{\sigma_0} - \frac{\sigma_0}{\sigma_0 + h}\right|M.
\end{equation}
We apply \eqref{eq.129} with $ \vartheta = \sigma_0/(\sigma_0 + h) $ and 
$ \sigma = \sigma_0 + h $:
\[
\I(\sigma_0)\leq \frac{\sigma_0}{\sigma_0 + h}\I(\sigma_0 + h) + \frac{1}{2}\left|
\frac{\sigma_0 + h}{\sigma_0} - \frac{\sigma_0}{\sigma_0 + h}\right| M.
\]
Then
\begin{equation}
\label{eq.131}
\I(\sigma_0 + h) - \I(\sigma_0)\geq\frac{h}{\sigma + h}I(\sigma_0 + h) - \frac{1}{2}
\left|\frac{\sigma_0}{\sigma_0 + h} - \frac{\sigma_0 + h}{\sigma_0}\right|M.
\end{equation}
From \eqref{eq.130} and \eqref{eq.131} we obtain that 
$ \I(\sigma_0 + h) - \I(\sigma_0)\to 0 $ as $ h\to 0 $. 
Therefore $ \I $ is continuous at $ \sigma_0 $.

\eqref{prop.2.3}. From (iii) of \cite[Proposition~2]{GG17} and \eqref{R2}, it follows that 
there exist $ \lambda_* > 0 $ such that, for every $ \lambda > \lambda_* $, the functional 
$ \E_S $ achieves negative values on $ S(\lambda) $. Therefore, if $ \sigma > m\lambda_* $
we choose $ u $ in $ S(\frac{\sigma}{m}) $ such that $ \E_S(u) $ is negative. 
From \eqref{eq.41} we have $ \Er(u,m) = \sigma m + \E_S(u) < \sigma m $. Then, we can
define
\[
\sigma_* := \inf\{\sigma > 0\mid I(\sigma) < \sigma m\}.
\]
If $ \sigma_* = 0 $, then the proof is concluded. Then we look at the case $ \sigma_* > 0 $.
From \eqref{prop.2.4} $ \I(\sigma_*) = \sigma_* m $. 
If $ 0 < \sigma\leq\sigma_* $ from \eqref{prop.2.1} and \eqref{prop.2.2} it follows that 
$ \I(\sigma) = \sigma m $. Now suppose that $ \sigma_* > 0 $. We show that $ \GSc $ is empty if 
$ 0 < \sigma\leq\sigma_* $. On the contrary, let $ \Phi $ be a minimum of 
$ \E $ on $ \M_\sigma $. From \eqref{prop.1.1} of Proposition~\ref{prop.1} and \eqref{eq.132}
$ (u,\omega) := P(\Phi) $ is a minimum of $ \Er $ on 
$ M_\sigma^* $. We apply \eqref{eq.129} to the constant sequence of 
$ (u_n,\omega_n) := (u,\omega) $ and $ \vartheta := \frac{\sigma_*}{\sigma} $. Then
\[
\begin{split}
I(\sigma_*)
&\leq \frac{\sigma_*}{\sigma} \I(\sigma) - 
\frac{1}{2}\left(\frac{\sigma_*}{\sigma} - \frac{\sigma}{\sigma_*}\right)\|u'\|_{L^2}^2 \\
&= 
\sigma_* m - \frac{1}{2}\left(\frac{\sigma_*}{\sigma} - 
\frac{\sigma}{\sigma_*}\right)\|u'\|_{L^2}^2.
\end{split}
\]
Since $ \I(\sigma_*) = \sigma_* m $ we have either $ \sigma = \sigma_* $, already ruled out
by the assumptions of \eqref{prop.2.3}, or $ \|u'\|_{L^2}^2 = 0 $. Then $ u $ is the zero function on $ (-\infty,+\infty) $, a conclusion which contradicts $ \omega\|u\|_{L^2}^2 = \sigma > 0 $. 
\end{proof}
\begin{remark}
In \eqref{prop.2.3} the case $ \sigma = \sigma_* $ has been intentionally left open.
\end{remark}
\begin{lemma}
\label{lem.concentration-compactness}
Suppose that \eqref{R1}, \eqref{R2} and \eqref{R3} are satisfied.
Then there exist $ \sigma_* $ such that for every $ \sigma > \sigma_* $
and every sequence $ (\Phi_n) $ in $ \X $ satisfying
\[
\Ch(\Phi_n) \to\sigma,\quad \E(\Phi_n)\to I(\sigma),
\]
there exist 
$ \Phi $ in $ \X $ such that a subsequence of $ (\Phi_n(\cdot +y_n)) $ converges to $ \Phi $ in
the metric of $ \X $.
\end{lemma}
\begin{proof}
For $ \sigma_* $ we choose the one defined in \eqref{prop.2.3} of Proposition~\ref{prop.2}.
We set $ \Phi_n := (\phi_n,\phi_{n,t}) $ and $ (u_n,\omega_n) := P(\Phi_n) $.
Since the sequence of $ \Er(u_n,\omega_n) = E(\Phi_n) $ is bounded, both $ (\phi_n,\phi_{n,t}) $ and $ (u_n,\omega_n) $ 
are bounded by \eqref{prop.1.3} of Proposition~\ref{prop.1}. Therefore, 
$ (u_n,\omega_n) := P(\Phi_n) $ is bounded as well. Up to extract a subsequence, we can
suppose that there exist $ \lambda $ such that 
\[
\|u_n\|_{L^2}^2\to\lambda,\quad\omega_n\to\frac{\sigma}{\lambda}.
\]
We can also suppose that $ \Es(\phi_n) $ converges to $ \inf_{S(\lambda)} (\E_S) $. Otherwise, given
$ \psi $ in $ S(\lambda) $ such that $ \Es(\psi) < \liminf_{n\to\infty}\Es(\phi_n) $, from \eqref{eq.41}
it would follow that $ \Er(\phi,\sigma/\lambda) < \I(\sigma) $. Moreover, $ \inf_{S(\lambda)}(\E_S) < 0 $. Otherwise,
from \eqref{eq.41}, we would obtain $ I(\sigma)\geq\sigma m $, which contradicts the properties of
$ \sigma_* $ established in \eqref{prop.2.3} of Proposition~\ref{prop.2}. Therefore, we can apply concentration properties of the variational setting $ (\Es,S(\lambda)) $ provided in \cite[Lemma~2.3]{GG17}. 
Then there exist $ (y_n) $ in $ \R $ and $ \phi $ in $ H^1(\R;\C) $ such that, up
to extract a subsequence,
\begin{equation}
\label{eq.110}
\phi_n (\cdot + y_n)\to\phi\text{ in } H^1 (\R;\C).
\end{equation}
We claim that a subsequence of $ \phi_{n,t}(\cdot + y_n) $ converges strongly in $ L^2 $.
In fact, up to extract a subsequence, there exist $ \phi_t $ in $ L^2 $ such that
$ \phi_{n,t}(\cdot + y_n)\wk\phi_t $ in $ L^2 (\R;\C) $. From \eqref{eq.110} it 
follows that
\begin{equation*}
\sigma = \lim_{n\to\infty} C(\phi_{n},\phi_{n,t}) = 
\lim_{n\to\infty} C(\phi_{n}(\cdot + y_n),\phi_{{n},t} (\cdot + y_n)) = 
C(\phi,\phi_{t}).
\end{equation*}
Therefore $ (\phi,\phi_t) $ is in $ \M_\sigma $. 
Now,
\[
\begin{split}
\I(\sigma) &= E(\phi_{n},\phi_{{n},t}) + o(1) = 
\E(\phi_n(\cdot + y_n),\phi_{n,t}(\cdot + y_n)) + o(1) \\
&=  
\frac{1}{2}\|\phi_{{n},t}(\cdot + y_n) - \phi_{t}\|_{L^2}^2 + \E(\phi,\phi_t) + o(1)\\ 
&+ \frac{m^2}{2}\|\phi_{{n}}(\cdot + y_n) - \phi\|_{L^2}^2 
+ \E_S (\phi_n(\cdot + y_n) - \phi)\\
&\geq \frac{1}{2}\|\phi_{{n},t}(\cdot + y_n) - \phi_{t}\|_{L^2}^2 + \I(\sigma) + o(1).
\end{split}
\]
The third equality follows \cite[Theorem~2]{BL83} and \eqref{eq.49}. The fourth inequality
follows \eqref{R1} and the fact that $ (\phi,\phi_{t}) $ belongs to $ M_\sigma $.
Taking the limit, we obtain
\[
\lim_{n\to\infty}\|\phi_{n,t}(\cdot + y_n) - \phi_{t}\|_{L^2}^2 = 0,
\]
which, together with \eqref{eq.110}, proves convergence stated in the lemma.
\end{proof}
\begin{remark}
In \cite{BF14}, the author devised an abstract framework to prove the concentration-compactness
of minimizing sequences of several variational problem. On this occasion we could not apply the result of \cite[Theorem~21]{BF14} to prove Lemma~\ref{lem.concentration-compactness} because of the lack of the assumption \cite[(EC-3).(i)]{BF14}, requiring $ \Ch(\Phi)\neq 0 $ whenever $ \Phi\neq 0 $. One can choose $ \Phi = (\phi,\phi) $ with $ \phi $ in $ H^1(\R;\R) $ as an example.
\end{remark}
\begin{cor}
\label{thm.min}
For every $ \sigma > \sigma_* $ the set $ \GSc $ is non-empty.
\end{cor}
\begin{proof}
Let $ \Phi_n := (\phi_n,\phi_{n,t}) $ be a minimizing sequence in $ \M_\sigma $. Since $ \sigma > \sigma_* $
the assumptions of Lemma~\ref{lem.concentration-compactness} are
satisfied. Then, there exist $ (y_n) $ in $ \R $
and $ \Phi : = (\phi,\phi_t) $ in $ \M_\sigma $ such that, up to extract
a subsequence, $ \Phi_n(\cdot + y_n)\to \Phi\text{ in } \X $.
From \eqref{prop.1.5} of Proposition~\ref{prop.1}, 
\[
I(\sigma) = \E(\Phi_n) = \E(\Phi_n(\cdot + y_n))\to \E(\Phi).
\]
Then $ \Phi $ is a minimum of $ \E $ on $ \M_\sigma $. 
\end{proof}
\begin{theorem}
\label{thm.minima-structure}
For every $ \sigma > \sigma_* $, given $ \Phi $ in $ \GSc $
there exist $ m > \omega > 0 $, a radially decreasing function $ R $ of class 
$ H^1 _{r,+} (\R;\R)\cap\Gamma_\sigma $ and $ (y,z) $ in $ \R\times S^1 $ such that
\[
\Phi = (zR(\cdot + y),-i\omega z R(\cdot + y)).
\]
\end{theorem}
\begin{proof}
Let $ \Phi $ be a minimum of $ \E $ on $ \M_\sigma $. We use the notation $ (u,\omega) := P(\Phi) $. By
definition of $ P $, the function $ u $ is non-negative.
From \eqref{prop.1.1} of Proposition~\ref{prop.1}, the point $ (u,\omega) $ belongs
to $ M_\sigma $. Then $ \Ir(\sigma)\leq\Er(u,\omega)\leq \E(\Phi) = \I(\sigma) $. 
From \eqref{eq.132} $ \I(\sigma) = \Ir(\sigma) $.
Therefore $ (u,\omega) $ is a minimum of $ \Er $ on $ M_\sigma^* $. Thus,
there exist $ \eta $ such that $ d\Er(u,\omega) = \eta d\Chr(u,\omega) $. 
We apply this equality between linear functionals to vectors of the form $ (v,0) $ and $ (0,1) $ and obtain
\begin{gather*}
\displaystyle\nint u'v'dx + (m^2 + \omega^2) \nint uvdx + \nint\G'(u)vdx  = 2\eta\omega\nint uvdx\\
\displaystyle\omega\|u\|_{L^2}^2 = \eta\|u\|_{L^2}^2.
\end{gather*}
Since $ \Chr(u,\omega) = \sigma > 0 $, the $ L^2 $ norm of $ u $ is different from zero. Therefore, from the
second equation we obtain $ \eta = \omega $; after the substitution in the first equation
we obtain
\begin{equation*}
\displaystyle\nint u' v'dx + (m^2 - \omega^2) \nint uvdx + \nint\G'(u)vdx  = 0
\end{equation*}
for every $ v $ in $ H^1 (\R;\R) $. By elliptic regularity,
\begin{equation}
\label{eq.56}
u''(x) - (m^2 - \omega^2) u(x) - \G'(u(x)) = 0.
\end{equation}
We multiply the equation by $ 2u' $. Then, there exist $ d $ in $ \R $ such that 
\begin{equation}
\label{eq.24}
u'(x)\sp 2 - (m^2 - \omega^2) u(x)\sp 2 - 2\G(u(x))\equiv d.
\end{equation}
On the left side we have a sum of $ L^1 $
functions. Therefore $ d = 0 $. Integrating
on $ \mathbb{R} $, we obtain
\begin{equation*}
\int_{-\infty}^{+\infty} u'(x)\sp 2 dx - (m^2 - \omega^2)\nint u(x)^2 dx - 2\int_{-\infty}^{+\infty} \G(u(x))dx = 0.
\end{equation*}
Since $ (u,\omega) $ is a minimum, the equality above becomes
\[
\nint u'(x)^2 dx + \omega^2\nint u(x)^2dx = \Er(u,\omega) = I(\sigma).
\]
By \eqref{prop.2.3} of Proposition~\ref{prop.2}, we have
\[
\nint u'(x)^2 dx + \omega^2\nint u(x)^2dx < \sigma m = m\omega\nint u(x)^2dx.
\]
Therefore $ \omega < m $. Since $ u $ is $ H\sp 1 $ it is also $ L\sp{\infty} $.
From \eqref{eq.24} and the continuity of $ \G $, the function
$ |u'| $ is bounded. Since $ u $ is in $ L^2 $, we have
\begin{equation*}
\lim_{|x|\to\infty} u (x) = 0.
\end{equation*}
Since $ u $ is also $ C^2 (\R) $, it satisfies 
condition \cite[(6.1)]{BL83a}. From \eqref{R3}, the differential
equation \eqref{eq.24} satisfies \cite[(6.2)]{BL83a} of \cite[Theorem~5]{BL83a}.
Therefore, according to the quoted theorem, there exist $ y $ in $ \R $ such 
that
\begin{equation}
\label{eq.57}
u(x + y) = R(x)
\end{equation}
where $ R $ solves \eqref{eq.56}, is positive, even and decreasing with respect to
the origin. That is $ R $ is in $ H^1 _{r,+} (\R;\R) $. 
Since $ \E(\Phi) = \Er(u,\omega) = I(\sigma) $, in
\eqref{eq.10} we have a chain of inequalities with the same values at the
endpoints. Therefore, all the intermediate inequalities are equalities. Then, we have
\begin{equation*}
\nint |\phi'(x)|^2 dx = \nint ||\phi|' (x)|^2dx.
\end{equation*}
This is a particular case of the Convex Inequality for Gradient \cite[Theorem~7.8,\ p.,\ 177]{LL01}, where the equality holds.
In \cite[Lemma~5.1]{Gar12}, we showed that if $ |\phi| $ is continuous and positive everywhere, there exist a complex
number $ z $ such that $ |z| = 1 $ and 
\begin{equation}
\label{eq.58}
\phi(x) = z|\phi(x)|.
\end{equation}
Comparing the first and second line of \eqref{eq.10}, we also obtain the equality
\begin{equation*}
\|\phi_t\|_{L^2} ^2 = \frac{|\Ch(\Phi)|^2}{\|\phi\|_{L\sp 2}\sp 2}.
\end{equation*}
This is a case of the Cauchy-Schwarz inequality where the equality holds. In fact, it reads
\[
\|\phi_t\|_{L^2} ^2 \|i\phi\|_{L^2}^2 = |(i\phi,\phi_t)_{L^2}|^2
\]
Therefore, since $ \phi\neq 0 $, there exist $ \beta $ in $ \R $ such that 
\begin{equation}
\label{eq.113}
\phi_t(x) = i\beta\phi(x)
\end{equation}
for every $ x $ in $ (-\infty,+\infty) $. Also, since $ \Ch(\Phi) = \Chr(|\phi|,\omega) = \sigma $, we have
\begin{equation}
\label{eq.59}
-\text{Im}\nint\phi_t\overline{\phi} dx = \omega\text{Re}\nint\phi\overline{\phi}dx.
\end{equation}
From \eqref{eq.113}, we obtain 
\[
\begin{split}
&\text{Im}\nint\phi_t\overline{\phi} dx = \text{Im}\nint i\beta\phi\overline{\phi}dx \\
= &\beta\text{Im}\,i\nint\phi\overline{\phi}dx = \beta\text{Re}\nint\phi\overline{\phi}dx.
\end{split}
\]
From the equality above and \eqref{eq.59}, we obtain $ \beta = -\omega $. Therefore,
\begin{equation}
\label{eq.60}
\phi_t (x) = -i\omega\phi(x)
\end{equation}
for every $ x $ in $ (-\infty,+\infty) $. From \eqref{eq.58}, \eqref{eq.113}, \eqref{eq.60} and \eqref{eq.57}, it follows
\[
\begin{split}
\Phi(x) &= (\phi(x),\phi_t (x)) = (\phi(x),-i\omega\phi(x)) \\
&= (zu(x),-i\omega zu(x)) = (zR(x + y),-i\omega zR(x + y))
\end{split}
\]
which concludes the proof.
\end{proof}
\begin{remark}
In order to obtain \eqref{eq.58}, we did not use the converse of the Convex Inequality for Gradients of \cite{LL01}. In fact, in the version provided by the authors, it is required that either $ \text{Re}(\phi) $ is positive everywhere
or $ \text{Im}(\phi) $ is positive everywhere, an information that it is not available at this point of the proof, even
though it is true \textsl{a posteriori}. However, here we know that 
$ \text{essinf}_\Omega|\phi| > 0 $ 
for every bounded set $ \Omega $. A proof of the equality \eqref{eq.58} with this assumption (instead of $ \mathrm{Re}(\phi) > 0 $) is in \cite[Lemma~5.1]{Gar12}. Another proof has been
provided in \cite[Proof~of~Theorem~2.8]{BBBM10} using a lifting map, that is a function $ S $ such that $ \phi/|\phi| = e^{iS(x)} $ with $ S $ in $ W^{1,1} _{loc} (\R^n) $, \cite{BBM00}.
We preferred to rely on \cite[Lemma~5.1]{Gar12}, because using the regularity of $ \phi/|\phi| $ directly seemed to us more straightforward.
\end{remark}
We conclude this section with the proof of Theorem~\ref{thm.ground-state-stability}.
\begin{proof}[Proof~of~Theorem~\ref{thm.ground-state-stability}]
Suppose that $ \Gamma_\sigma $ is not stable. Then there are sequences 
$ (\Phi_n) $, $ (t_n) $ and $ \varepsilon_0 > 0 $ such that
\begin{equation}
\label{eq.135}
\mathrm{dist}(\Phi_n,\Gamma_\sigma)\to 0,\quad
\mathrm{dist}(U(t_n,\Phi_n),\Gamma_\sigma)\geq\varepsilon_0.
\end{equation}
We set $ \Psi_n := U(t_n,\Phi_n) $.
By \eqref{prop.1.5} of Proposition~\ref{prop.1} $ \E(\Phi_n)\to\I(\sigma) $ and
$ \Ch(\Phi_n)\to\sigma $. From \eqref{eq.101}, we have $ \E(\Psi_n) = \E(\Phi_n) $ and 
$ \Ch(\Psi_n)) = \Ch(\Phi_n) $. Therefore, $ \E(\Psi_n))\to\I(\sigma) $ and
$ \Ch(\Psi_n))\to\sigma $. By Lemma~\ref{lem.concentration-compactness}, up to extract
a subsequence, we can suppose that there exist $ (y_n)\subseteq\R $ and $ \Psi $ in 
$ \Gamma_\sigma $ such that $ \Phi_n(\cdot + y_n)\to\Psi $. Therefore, 
$ \|\Psi_n - \Phi(\cdot - y_n)\|_\X\to 0 $ implying that $ \dist(\Psi_n,\Gamma_\sigma)\to 0 $
and giving a contradiction with \eqref{eq.135}.
\end{proof}
\section{The double power case: the cardinality of $ K_\sigma $}
\label{sect.du}
So far our conclusions hold for general assumptions on the non-linearity $ \G $. 
In this section we address the double power non-linearity \eqref{eq.DP}.
In this section, for every $ \sigma $ we count the number of minima of $ \E $ constrained to 
$ \M_\sigma $ which are real-valued and radially decreasing. We set: 
\begin{equation*}
M_{\sigma,r}^* :=  \{(u,\omega)\in\R\times H^1_r (\R;\R)\mid\Chr(u,\omega) = \sigma\}.
\end{equation*}
\begin{equation*}
\Ers\colon M_{\sigma,r}^*\to\R,\quad (u,\omega)\mapsto \Er(u,\omega).
\end{equation*}
We use a similar approach to the one in \cite{GG17}. We define:
\begin{equation}
\label{eq.81}
V(s) := - \frac{2G(s)}{s^2} = 2as^2 - 2bs^4.
\end{equation}
Given $ m > \omega > 0 $, we define
\begin{equation}
\label{eq.79}
R_* (\omega) := \inf\{s > 0\mid V(s) = m^2 - \omega^2\}
\end{equation}
as long as the set is non-empty and
\[
s_* = \left(\frac{a}{2b}\right)^{\frac{1}{2}},\quad V(s_*) = \sup(V) = \frac{a^2}{2b}
\]
the unique positive local maximum. We also define
\begin{equation}
\label{eq.80}
\omega_* := (m^2 - \sup(V))^{\frac{1}{2}}.
\end{equation}
The next proposition contains a list of properties of $ R_* $ that we will
use in this section. We do not provide a detailed proof of them, as they follow
from the inspection of the graph of $ V $ and the Implicit Function Theorem.
\begin{proposition}
\label{prop.3}
If $ \G $ is a double power non-linearity as in \eqref{eq.DP}, then
\[
\{s \geq 0\mid V(s) = m^2 - \omega^2\}\neq\emptyset\iff \omega\in [\omega_*,m].
\]
Then $ R_* $ is defined on the interval $ [\omega_*,m] $. Moreover, 
\begin{enumerate}
\item $ R_* $ is a smooth function from $ (\omega_*,m) $ to $ (0,s_*) $
\item\label{prop.3.2} $ V(R_* (\omega)) = m^2 - \omega^2 $
\item $ R_* '(\omega) < 0 $ for every $ \omega $ in $ (\omega_*,m) $
\item\label{prop.3.4} $ R_*(\omega_*) = s_* $ and $ R_*(m) = 0 $.
\end{enumerate}
\end{proposition}
\begin{proof}
From $ \omega_* \leq \omega \leq m $, it follows that 
$ m^2 - \omega_*^2 \geq m^2 - \omega^2 \geq 0 $.
By definition of $ \omega_* $, there holds $ \sup(V) \geq m^2 - \omega^2 \geq 0 $. Since
$ V $ is continuous and $ V(0) = 0 $, the image of $ V $ contains $ m^2 - \omega^2 $. 
\smallskip

(i). For every $ \omega $ in $ (\omega_*,m) $, the set in \eqref{eq.79} contains
exactly two points $ s_1 < s_* < s_2 $ and $ R_* (\omega) = s_1 $. Since $ V $ is
smooth and $ V'(s_1) > 0 $, the regularity of $ R_* $ follows by applying the
Implicit Function Theorem to $ g(\omega,s) := V(s) - m^2 + \omega^2 $.
\smallskip

(ii). This follows from the definition of $ R_* $. (iii). Taking the derivative with
respect to $ \omega $ in \eqref{prop.3.2}, we obtain $ V'(R_* (\omega))R_* '(\omega) = -2\omega $
whence $ R_* ' (\omega) < 0 $.\smallskip

(iv). From \eqref{eq.80}, $ V(s_*) = m^2 - \omega_*^2 $. Since $ \sup(V) $ is achieved
in a unique point, $ R_* (\omega_*) = \sqrt{a/2b} $. If $ \omega = m $, then $ R_* (m) = 0 $,
because $ V $ has exactly two zeroes on $ [0,+\infty) $ and the origin is one of these.
\end{proof}
Let $ R_\omega $ be the solution of the initial value problem 
\begin{equation}
\label{eq.IVP}
R_\omega ''(x) = G'(R_\omega (x)) + (m^2 - \omega^2) R_\omega(x),
\quad R_\omega '(0) = 0,\quad R_\omega(0) = R_* (\omega).
\end{equation}
\begin{proposition}
\label{prop.4}
For every $ \omega $ in $ (\omega_*,m) $.
\begin{enumerate}
\item\label{prop.4.1} $ R_\omega $ is positive
\item\label{prop.4.2} $ R_\omega $ is even and radially strictly 
decreasing with respect to the origin
\item\label{prop.4.3} $ R_\omega $ and $ R_\omega' $ are exponentially decaying as 
$ |x|\to\infty $.
\end{enumerate}
In particular, $ R_\omega $ is in $ H^1 _{r,+} (\R;\R) $ and it is radially decreasing.
\end{proposition}
\begin{proof}
From \eqref{eq.81} and \eqref{eq.79}, 
the point $ R_* (\omega) $ is the first positive zero of the non-linear term
\[
F(s) := -G(s) - \frac{1}{2}(m^2 - \omega^2)s^2.
\]
Therefore, by applying \cite[Theorem~5]{BL83a} with $ \zeta_0 := R_* (\omega) $, we obtain 
\eqref{prop.4.1} and \eqref{prop.4.2}, which follow from
(i), (ii) and (iv) of the quoted theorem. Still using the notations of the same
paper
\[
\lim_{s\to 0}\frac{F'(s)}{s} = - (m^2 - \omega^2).
\]
From Theorem~\ref{thm.minima-structure} the limit above is negative. Therefore the assumptions of \cite[Remark~6.3]{BL83a} are satisfied and both
$ R_\omega $ and $ R_\omega' $ have exponential decay as $ |x| $ diverges. Then
$ R_\omega $ is in $ H^1 _{d,+} (\R;\R) $.
\end{proof}
From Proposition~\ref{prop.4}, we have two well-defined functions:
\begin{equation}
\label{eq.83}
\begin{split}
\sigma&:(\omega_*,m)\to (0,+\infty),\quad \sigma(\omega) := \omega\|R_\omega\|_{L^2}^2 = \Chr(R_\omega,\omega)\\
e&:(\omega_*,m)\to (0,+\infty),\quad e(\omega) := \Er(R_\omega,\omega).
\end{split}
\end{equation}
\begin{remark}
The double power non-linearity defined in \eqref{eq.DP} satisfies \eqref{R2} and \eqref{R3}.
In fact, $ \G $ achieves negative values in a neighbourhood of the origin; the inequality
in \eqref{R3} is satisfied with powers $ p = 4 $ and $ q = 6 $, and $ \G(0) = 0 $.
Assumption \eqref{R1} is satisfied only for appropriate choices of $ a, b $ and $ m $.
In fact, the infimum of the function
\begin{equation*}
\frac{\V(s)}{s^2} = \frac{m^2}{2} - as^2 + bs^4
\end{equation*}
is positive if and only if it achieves a positive value at the minimum. Equivalently,
$ \tau(a,b,m) > 1 $ which is the function introduced in \eqref{eq.127}. Finally, it satisfies (\hyperlink{R4}{G4}). In fact, the non-linearity is of class $ C^\infty $ with 
$ \G'(0) = 0 $, which implies the existence of local solutions and \eqref{R1} provides
a priori bounds for local solutions, \cite{GV85}.
\end{remark}
Hereafter, we will assume that $ \tau > 1 $.
\begin{lemma}
\label{lem.decreasing-charge}
There exist $ 1 < \tau_* $ such that for every $ a,b $ and $ m $
\begin{enumerate}
\item if $ \tau(a,b,m)\geq \tau_* $
the function $ \sigma $ is strictly decreasing on
$ (\omega_*,m) $. If $ \tau = \tau_* $, it has a unique saddle point
$ \omega_s (a,b,m) $ 
\item\label{lem.decreasing-charge.2}
if $ 1 < \tau < \tau_* $, there are two critical points
$ \omega_{m} (a,b,m) < \omega_{M} (a,b,m) $
which are, respectively, a local minimum and a local maximum.
\end{enumerate}
\end{lemma}
\begin{proof}
From \eqref{prop.4.3} of Proposition~\ref{prop.4}, if we multiply \eqref{eq.IVP} by $ 2R_\omega' $ and integrate, we obtain
\[
R_\omega '(x)^2 = (m^2 - \omega^2) R_\omega (x) ^2 + 2G(R_\omega(x)).
\]
By \eqref{prop.4.2} of Proposition~\ref{prop.4}, we have
\[
R_\omega '(x) = -\sqrt{(m^2 - \omega^2) R_\omega (x)^2 + 2G(R_\omega (x))}.
\]
Since $ R_\omega $ is even, we can restrict to the integration on the interval
$ (0,+\infty) $. Therefore from \eqref{eq.83},
\begin{equation}
\label{eq.84}
\begin{split}
\sigma(\omega) &= 2\omega\int_0^\infty  R_\omega (x)^2 dx \\
&= 2\omega \int_0^\infty
\frac{R_\omega (x)^2 R_\omega'(x)dx}{-\sqrt{(m^2 - \omega^2) R_\omega (x)^2 + 2G(R_\omega(x))}}\\
&= 2\omega \int_0^\infty
\frac{R_\omega (x) R_\omega'(x)dx}{-\sqrt{m^2 - \omega^2 - V(R_\omega(x))}}\\
&= 2\omega \int_0^\infty
\frac{R_\omega (x) R_\omega'(x)dx}{-\sqrt{m^2 - \omega^2 - 2a R_\omega(x)^2 + 2bR_\omega(x)^4}}\\
&= \omega\int_{0}^{R_*(\omega)^2} \frac{ds}%
{\sqrt{m^2 - \omega^2 - 2as + 2bs^2}}.
\end{split}
\end{equation}
From \eqref{prop.3.2} of Proposition~\ref{prop.3}, we have
\begin{equation}
\label{eq.115}
m^2 - \omega^2 = 2aR_* (\omega)^2 - 2bR_* (\omega)^4
\end{equation}
which gives
\begin{equation}
\label{eq.85}
\left|R_* (\omega)^2 - \frac{a}{2b}\right|^2  = 
\frac{a^2}{4b^2} - \frac{m^2 - \omega^2}{2b} = \frac{a^2}{4b^2}(1 - \alpha^2(\omega))
\end{equation}
where
\begin{equation}
\label{eq.117}
\alpha(\omega) = \left(\frac{2b(m^2 - \omega^2)}{a^2}\right)^{\frac{1}{2}}.
\end{equation}
In order to find a suitable integration by substitution, we 
rearrange the argument of the square root in \eqref{eq.84}. From \eqref{eq.115},
\begin{equation}
\label{eq.116}
\begin{split}
m^2 - \omega^2 - 2as + 2bs^2 &= 2aR_* (\omega)^2 - 2bR_* (\omega)^4 - 2as + 2bs^2\\
&= 2b\left(\bigg(s - \frac{a}{2b}\bigg)^2 - \bigg(R_*(\omega)^2 - \frac{a}{2b}\bigg)^2\right).
\end{split}
\end{equation}
From \eqref{eq.84}, we can continue as 
\[
\begin{split}
&\frac{\omega}{\sqrt{2b}}
\int_0^{R_*(\omega)^2}\frac{ds}{\sqrt{(s - \frac{a}{2b})^2 - 
(R_*(\omega)^2 - \frac{a}{2b})^2}} \\
=& \frac{\omega}{\sqrt{2b}}
\Bigg[\ln\bigg|s - \frac{a}{2b} + \sqrt{\left(s - \frac{a}{2b}\right)^2 - 
\left(R_*(\omega)^2 - \frac{a}{2b}\right)^2}\bigg|\Bigg]_0^{R_* (\omega)^2}\\
=& \frac{\omega}{\sqrt{2b}} \ln\bigg|R_*(\omega)^2 - \frac{a}{2b}\bigg|
- \frac{\omega}{\sqrt{2b}} 
\ln\bigg| - \frac{a}{2b} + \sqrt{\frac{a^2}{4b^2} - 
\left(R_*(\omega)^2 - \frac{a}{2b}\right)^2}\bigg|
\end{split}
\]
using \eqref{eq.116}.
The integral can be evaluated by means of a hyperbolic trigonometric function substitution. 
We obtain
\[
\begin{split}
& \frac{\omega}{\sqrt{2b}}\ln\bigg|\frac{a}{2b}\sqrt{1 - \alpha^2(\omega)}\bigg|
-\frac{\omega}{\sqrt{2b}} \ln\bigg| - \frac{a}{2b} + \frac{a}{2b}\alpha(\omega)\bigg|\\
=& \frac{\omega}{2\sqrt{2b}}\ln(1 - \alpha^2(\omega))
-\frac{\omega}{2\sqrt{2b}} \ln(1 - \alpha(\omega))^2.
\end{split}
\]
Therefore, 
\[
\sigma(\omega) = \frac{\omega}{2\sqrt{2b}}\ln\left(\frac{1 + \alpha(\omega)}{1 - \alpha(\omega)}\right).
\]
In order to study the sign of the derivative of $ \sigma $, we represent it as
the composite function of $ \alpha $, which is decreasing and surjective
from the interval $ (\omega_*,m) $ to $ (0,1) $. This follows from \eqref{eq.85} and
\eqref{prop.3.4} of Proposition~\ref{prop.3}. From \eqref{eq.117} and \eqref{eq.127},
\[
\omega = \frac{a}{\sqrt{2b}}\sqrt{\tau - \alpha^2}.
\]
Therefore,
\begin{equation}
\label{eq.119}
\sigma(\omega) = \frac{a}{4b} k_1 (\alpha(\omega)),\quad\sigma'(\omega) = 
\frac{a}{4b} k_1'(\alpha(\omega))\alpha'(\omega)
\end{equation}
where
\[
k_1 (\alpha) = \sqrt{\tau - \alpha^2}
\ln\left(\frac{1 + \alpha}{1 - \alpha}\right),\quad\alpha\in (0,1).
\]
We have
\[
k_1 '(\alpha) = -\frac{\alpha}{\sqrt{\tau - \alpha^2}}
\ln\left(\frac{1 + \alpha}{1 - \alpha}\right) + 
\sqrt{\tau - \alpha^2}
\cdot\frac{2}{1 - \alpha^2}.
\]
Then
\begin{equation}
\label{eq.87}
\begin{split}
k_1 '(\alpha)\sqrt{\tau - \alpha^2}\left(\frac{1 - \alpha^2}{2}\right) 
&= 
\tau - \left(\alpha^2 + \frac{\alpha - \alpha^3}{2}
\ln\left(\frac{1 + \alpha}{1 - \alpha}\right)\right) =: \tau - k_2(\alpha).
\end{split}
\end{equation}
From \eqref{eq.119} and \eqref{eq.87} it follows
\begin{equation}
\label{eq.120}
\sigma'(\omega) = \frac{a}{2b}\frac{\alpha'(\omega)}%
{(1 - \alpha(\omega)^2)\sqrt{\tau - \alpha(\omega)^2}} 
\big(\tau - k_2(\alpha(\omega))\big).
\end{equation}
We define
\[
\tau_* := \sup_{\alpha\in(0,1)} k_2(\alpha).
\]
The behaviour of $ k_2 $ at the endpoints is
\begin{equation}
\label{eq.134}
k_2(0) = 0,\quad\lim_{\alpha\to 1} k_2 = 1.
\end{equation}
As $ \alpha $ converges to 1 the function $ k_2 $ converges to 1, proving that
$ \tau_*\geq 1 $. We show that the properties of the critical points of $ \sigma $
are exactly how we stated in the lemma. Since $ \alpha'(\omega) < 0 $, always different from
zero, from \eqref{eq.119} and \eqref{eq.87} it is sufficient to restrict to the solutions to 
\begin{equation}
\label{eq.133}
k_2 (\alpha) = \tau_*,\quad\alpha\in (0,1).
\end{equation}
\subsubsection*{The case $ \tau > \tau_* $}
Since $ \tau > k_2 $ on $ (0,1) $ the function $ k_1' $ is always positive, then $ \sigma' $
is negative on $ (\omega_*,m) $ by \eqref{eq.87}. Therefore $ \sigma $ is strictly decreasing.
\subsubsection*{The case $ \tau = \tau_* $}
We can show that there is one solution to \eqref{eq.133}. In fact,
\[
k_2'(\alpha) = 3\alpha + \frac{1 - 3\alpha^2}{2}\ln\left(\frac{1 + \alpha}{1 - \alpha}\right).
\]
We have
\[
\lim_{\alpha\to 0} k_2' = 0,\quad\lim_{\alpha\to 1} k_2' = -\infty.
\]
Therefore $ k_2 $ is decreasing in a neighbourhood of 1. Then $ k_2 $ achieves its supremum
in the interior of $ [0,1] $ proving that $ \tau_* > 1 $ and that 
$ \tau_* - k_2 $ has at least one zero in $ (0,1) $. We show that the supremum is achieved
only once. On the contrary $ k_2' $ would have two zeroes in the interval $ (0,1) $.
In a neighbourhood of the origin $ k_2'\simeq 4\alpha $, that is, has positive sign. Then 
if $ k_2' $ has a two zeroes it must have a third one, unless in one of the two 
$ k_2 '' $ vanishes as well. In both cases $ k_2'' $ would have two zeroes in $ (0,1) $. However,
\[
k_2''(\alpha)= \frac{4 - 6\alpha^2}{1 - \alpha^2} - 3\alpha
\ln\left(\frac{1 + \alpha}{1 - \alpha}\right)
\]
is the sum of two strictly decreasing functions. Thus, only one zero is allowed to
exist for $ k_2'' $. Now let $ \alpha_s $ be the unique zero of $ k_2' $. Since $ \alpha $ is
bijective, there exist only one $ \omega_s $ such that $ \alpha(\omega_s) = \alpha_s $.
From \eqref{eq.120} with $ \tau = \tau_* $ it follows that $ \sigma'(\omega_s) = 0 $ and 
$ \sigma''(\omega_s) = 0 $, while $ \sigma' $ is negative
at any other point of the interval $ (\omega_*,m) $.
\subsubsection*{The case $ 1 < \tau < \tau_* $}
Since $ k_2 $ has a unique critical point where its maximum is achieved, any value in the interval $ [1,\tau_*) $ is achieved exactly two times from \eqref{eq.134}.
Let $ \alpha_1 < \alpha_2 $ be such that $ k_2 (\alpha_1) = k_2 (\alpha_2) $;
$ k_2 > \tau $ on $ (\alpha_1,\alpha_2) $ and
$ k_2 < \tau $ on $ (0,\alpha_1)\cup(\alpha_2,1) $. Let $ \omega_m $ and $ \omega_M $ be such that
$ \alpha(\omega_m) = \alpha_2 $ and $ \alpha(\omega_M) = \alpha_1 $. Then
$ \omega_m < \omega_M $ and from \eqref{eq.120}
$ \sigma' < 0 $ on $ (\omega_*,\omega_m)\cup(\omega_M,m) $ and
$ \sigma' > 0 $ on $ (\omega_m,\omega_M) $, just as we stated in the lemma.
\end{proof}
Hereafter, we will use the notation 
\begin{equation}
\label{eq.121}
\sigma_m := \sigma(\omega_m),\quad\sigma_M := \sigma(\omega_M).
\end{equation}
The next lemma shows that from the point of view of the multiplicity of minima, one can have different behaviours depending on the choice of $ a,b $ and $ m $ and the constraint 
$ M_{\sigma,r}^* $.
\begin{theorem}
\label{thm.uniqueness}
If $ 1 < \tau < \tau_* $ there exist $ \sigma_2 $ in $ (\sigma_m,\sigma_M) $ such that two positive 
critical points of $ \Ers $ on $ M_{\sigma,r}^* $ have different energy if 
$ \sigma\neq\sigma_2 $. If 
$ \sigma = \sigma_2 $, there are two minima.
\end{theorem}
\begin{proof}
For every $ \omega $ in $ (\omega_*,m) $ there holds
\begin{equation}
\label{eq.90}
e'(\omega)  = \omega\sigma'(\omega).
\end{equation}
In fact, from Theorem~\ref{thm.minima-structure}, $ (R_\omega,\omega) $ is a critical point of $ \Er $ on
$ \Mr_\sigma $ with Lagrange multiplier $ \omega $.
Since the function $ (\omega_*,m)\ni\omega\mapsto R_\omega $ is 
$ C^1 \big((\omega_*,m),H^1 _{r,+} (\R;\R)\big) $, we have
\[
\begin{split}
e'(\omega) &= \frac{d}{d\omega} \Er(R_\omega,\omega) = 
d\Er(R_\omega,\omega)[(\partial_\omega R_\omega,1)] \\
&= \omega
d\Chr(R_\omega,\omega)[(\partial_\omega R_\omega,1)] = \omega\frac{d}{d\omega}\Chr(R_\omega,\omega) = 
\omega\sigma'(\omega).
\end{split}
\]
For every $ \sigma $ in $ (\sigma_m,\sigma_M) $, there are three points 
$ \omega_1 (\sigma) < \omega_2(\sigma) < \omega_3 (\sigma) $ such that 
$ \sigma(\omega_i (\sigma)) = \sigma $. From Lemma~\ref{lem.decreasing-charge} and $ 1 < \tau < \tau_* $, the Mean Value Theorem 
forces $ \omega_1 (\sigma) < \omega_m < \omega_2 (\sigma) < \omega_M < \omega_3 (\sigma) $ 
(check Figure~\ref{fig.1}). Moreover, $ \omega_i $ are smooth functions on $ (\sigma_m,\sigma_M) $. 
We define
\[
g_1 (\sigma) :=  \int_{\omega_1(\sigma)}^{\omega_2 (\sigma)} (\sigma - \sigma(t))dt,
\quad g_2 (\sigma) :=  \int_{\omega_2(\sigma)}^{\omega_3(\sigma)} (\sigma(t) - \sigma)dt.
\]
Since $ (\omega_1 (\sigma),\omega_2 (\sigma))\subseteq (\omega_*,\omega_M) $ and
$ (\omega_2 (\sigma),\omega_3 (\sigma))\subseteq (\omega_M,m) $, both functions 
are positive. Moreover,
\begin{equation*}
g_1 (\sigma_m) = g_2 (\sigma_M) = 0,\quad g_1(\sigma_M),\ g_1(\sigma_m) > 0 .
\end{equation*}
Since $ \omega_1 (\sigma) < \omega_m < \omega_2 (\sigma) $, for $ \sigma' > \sigma $ we
have $ [\omega_1 (\sigma),\omega_2 (\sigma)]\subseteq [\omega_1 (\sigma'),\omega_2 (\sigma')] $
and $ [\omega_2 (\sigma),\omega_3 (\sigma)]\supseteq [\omega_2 (\sigma'),\omega_3 (\sigma')] $.
Then, $ g_1 $ is an increasing function and $ g_2 $ a decreasing function. Therefore,
$ g_1 - g_2 $ is a strictly increasing function  on the interval $ [\sigma_m,\sigma_M] $
attaining different signs at the endpoints. By the Intermediate Value Theorem, there exist
a unique $ \sigma_2 $ in $ (\sigma_m,\sigma_M) $ such that 
$ g_1 (\sigma_2) - g_2 (\sigma_2) = 0 $. From \eqref{eq.90} it follows that
\begin{equation}
\label{eq.99}
\begin{split}
e(\omega_3 (\sigma)) - e(\omega_1 (\sigma)) &= 
\int_{\omega_1 (\sigma)}^{\omega_3 (\sigma)} e'(t)dt =
\int_{\omega_1 (\sigma)}^{\omega_3 (\sigma)} t\sigma'(t)dt \\
&= -\int_{\omega_1 (\sigma)}^{\omega_3 (\sigma)} \sigma(t)dt + 
\sigma(\omega_3 (\sigma) - \omega_1 (\sigma)) \\
&= 
\int_{\omega_1 (\sigma_2)}^{\omega_3%
(\sigma_2)} (\sigma - \sigma(t))dt = (g_1 - g_2)(\sigma).
\end{split}
\end{equation}
for every $ \sigma $ in $ (\sigma_m,\sigma_M) $. Then
$ e(\omega_1 (\sigma_2)) = e(\omega_3 (\sigma_2)) $; if 
$ \sigma\neq\sigma_2 $, the last term of \eqref{eq.99} is different than 0, because 
$ \sigma_2 $ is the unique zero of $ g_1 - g_2 $ on $ [\sigma_m,\sigma_M] $. 
If $ \sigma\notin [\sigma_m,\sigma_M] $, there are not
critical points with the same charge, by \eqref{lem.decreasing-charge.2} of Lemma~\ref{lem.decreasing-charge}.
The equality \eqref{eq.99} does not hold for the pair 
$ \{\omega_1 (\sigma_2),\omega_2 (\sigma_2)\} $ or $ \{\omega_2 (\sigma_2),\omega_3 (\sigma_2)\} $, as \eqref{eq.99} would be equal to $ g_1 (\sigma_2) $ and
$ -g_2 (\sigma_2) $, respectively, a quantity different from zero in any case. 
Therefore, 
\[
\Er(R_{\omega_2 (\sigma_2)},\omega_2(\sigma_2)) > \Er(R_{\omega_i (\sigma_2)},\omega_i(\sigma_2))
\]
for $ i = 1,3 $. This allows us to conclude that there are exactly two positive minima if 
$ \sigma = \sigma_2 $ and one positive minimum if $ \sigma\neq\sigma_2 $. In fact, by \cite[Theorem~5]{BL83a} for fixed $ \omega $ there exist a unique positive and decaying solution 
to \eqref{eq.56}.
\end{proof}
\begin{center}
\begin{figure}[h!]
\centering\includegraphics[scale=0.7]{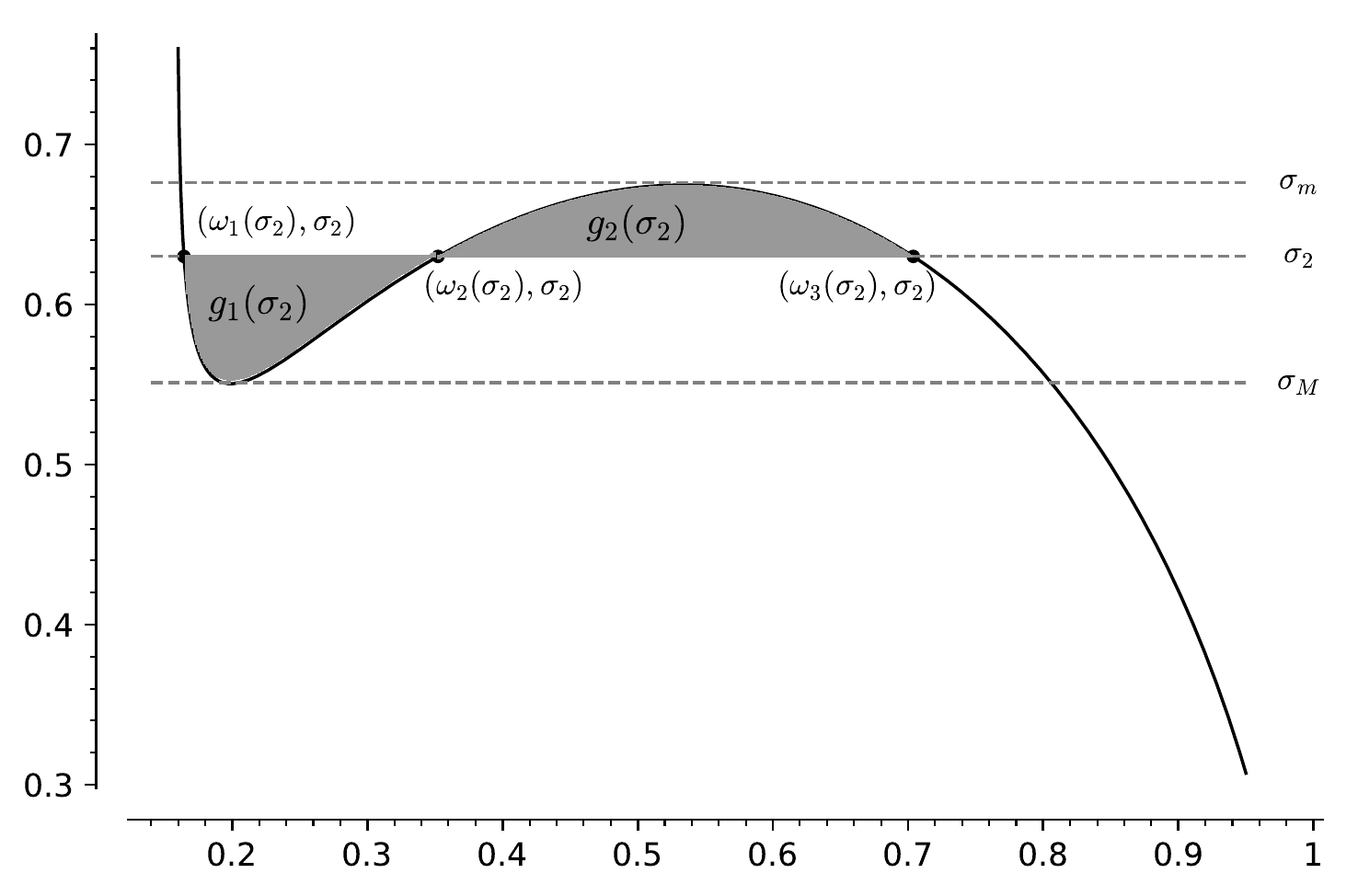}
\caption{Graph of $ \sigma(\omega) $ with $ 1 < \tau < \tau_* $. 
$ \sigma_2 $ is the level where the two shaded regions have the same area or, equivalently,
there are two minima.}
\end{figure}
\end{center}
\begin{remark}
\label{rem.regular-one-parameter}
A proof of the regularity of the one-parameter family defined in \eqref{eq.IVP} exists in dimension $ n\geq 3 $ as a result of \cite[Lemma~20]{SS85}. The only change needed in the quoted reference in order to obtain a proof tailored to our assumption
is $ \G'(u) $ is a function in $ L^2 $ (instead of $ L^{\frac{2n}{n + 2}} (\R^n) $ as in the quoted lemma). 
The authors also assume that the only $ H^1 _r (\R^n;\R) $ solution to $ \Delta v - 
\G''(R_\omega)v - (m^2 - \omega^2)v = 0 $
is the zero function. In our case ($ n = 1 $) we do not need such assumption as it follows
as a special outcome of the dimension one.
\end{remark}
\begin{remark}
\label{rem.minimum-level}
In the case \eqref{eq.DP} the threshold $ \sigma_* $ provided in \eqref{prop.2.3}
of Proposition~\ref{prop.2}, is zero. In fact, given $ \sigma > 0 $ one can choose an
element $ u $ in $ S(\sigma/m) $ such that $ \E_S (u) < 0 $. This follows, for instance,
from \cite[Lemma~5]{BBGM07}, under the assumption that
$ \G(s)\lesssim -s^{2 + \var} $ with $ 0 < \var < 4/n $ in
a neighbourhood of the origin (\cite[$ F_2 $]{BBGM07}), which is satisfied indeed with
$ \var = 2 $. Therefore, from \eqref{eq.41},
$ \Er(u,m) = \sigma m + \E_S(u) < \sigma m $.
\end{remark}
In the next theorem, which concludes this section, we will be able to count the number
of positive critical points of $ \Er $ on $ M_{\sigma,r}^* $ and to establish for each of them
whether it is a minimum or not. Let $ K_\sigma $ be the set defined in \eqref{eq.critical-set} and set
\[
Cr_{\sigma} := \{(R,\omega)\mid\exists\eta\text{ s.t. } d\Er(R,\omega) = \eta d\Chr(R,\omega)\}\cap
H^1 _{r,+} (\R;\R)
\]
which is the set of critical points of $ \Er $ constrained on $ \M_{\sigma,r}^* $.
\begin{theorem}
\label{thm.multiplicity-result}
For every $ a,b,m $ such that $ \tau(a,b,m) > 1 $, the number of positive critical points and minima
of $ \Er $ on $ M_{\sigma,r}^* $ behaves as follows:
\begin{enumerate}[(i)]
\item\label{thm.multiplicity-result.1} if $ \tau\geq\tau_* $ then 
$ |K_{\sigma}| = |Cr_{\sigma}| = 1 $
\item\label{thm.multiplicity-result.2} if $ 1 < \tau < \tau_* $ then there are two levels 
$ \sigma_m (a,b,m) < \sigma_M (a,b,m) $ such that 
\begin{enumerate}[\ (a)]
\item\label{thm.multiplicity-result.2.1} 
$ |K_{\sigma}| = |Cr_{\sigma}| = 1 $ if $ \sigma < \sigma_m $ or $ \sigma > \sigma_M $
\item\label{thm.multiplicity-result.2.2} 
$ |Cr_{\sigma}| = 2 $ and $ |K_{\sigma}| = 1 $ if $ \sigma\in\{\sigma_m,\sigma_M\} $
\item\label{thm.multiplicity-result.2.3}
 $ |Cr_{\sigma}| = 3 $ and $ |K_{\sigma}| = 1 $ if $ \sigma_m < \sigma < \sigma_M $ with $ \sigma\neq\sigma_2 $
\item\label{thm.multiplicity-result.2.4}
$ |Cr_{\sigma}| = 3 $ and $ |K_{\sigma}| = 2 $ if $ \sigma = \sigma_2 $. That is, there are \textsl{two} minima.
\end{enumerate}
\end{enumerate}
The level $ \sigma_2 $ is the level provided in Theorem~\ref{thm.uniqueness}.
\end{theorem}
\begin{proof}
From Remark~\ref{rem.minimum-level}, Corollary~\ref{thm.min} and 
Theorem~\ref{thm.minima-structure}, for every $ \sigma > 0 $ there exist at least one 
positive minimum.
\subsubsection*{The case $ \tau\geq\tau_* $}
From Lemma~\ref{lem.decreasing-charge}, this is the case where $ \sigma $ is a strictly
decreasing function. Suppose that there are two positive solutions 
$ (\omega_1,R_{\omega_1}) $
and $ (\omega_2,R_{\omega_2}) $ in $ M_{\sigma,r} $. Without loss of generality we can
assume that $ \omega_1\leq\omega_2 $. If $ \omega_1 = \omega_2 $, then $ R_{\omega_1} = R_{\omega_2} $
by the uniqueness result \cite[Proof~of~Theorem~5]{BL83a}. If $ \omega_1 < \omega_2 $, then
$ \sigma(\omega_1) > \sigma(\omega_2) $ gives a contradiction. Since there is only one
critical point, it is a minimum.
\subsubsection*{The case $ 1 < \tau < \tau_* $}
On the intervals
$ (0,\sigma_m) $ and $ (\sigma_M,+\infty) $ we use the same argument as the previous
case. Then, there is exactly one positive minimum at each level. 

If $ \sigma = \sigma_M $, the value $ \sigma $ is achieved at 
$ \omega_M $. From \eqref{eq.119} and \eqref{eq.117}
$ \sigma(\omega) $ diverges as $ \omega\to\omega_* $. Then $ \sigma_M $ is achieved in a second
point $ \omega_2 < \omega_M $. The existence of a third point would imply the existence of three
critical points for $ \sigma $ which is ruled out by 
Lemma~\ref{lem.decreasing-charge}. 
Therefore, $ |Cr_{\sigma}| = 2 $ and only one critical
point is a minimum, from Theorem~\ref{thm.uniqueness}.

If $ \sigma = \sigma_m $, the value $ \sigma $ is
achieved in $ \omega_m $. Since $ \omega_m $ is a local minimum, and $ \sigma(\omega) $
converges to 0 as $ \omega\to m $, the value $ \sigma $ is achieved in a second point
$ \omega_2 $. The existence of a third point would imply the existence of three
critical points for $ \sigma $ which is ruled out by 
Lemma~\ref{lem.decreasing-charge}. 
Then, $ |Cr_{\sigma}| = 2 $ and only one critical point is a minimum, from Theorem~\ref{thm.uniqueness}.

Finally, if $ \sigma_m < \sigma < \sigma_M $, there are three points in $ (\omega_*,m) $
such that $ \sigma(\omega) = \sigma $. More precisely, these points can be obtained
by applying the Intermediate Value Theorem to the function $ \sigma $ on the intervals
$ (\omega_*,\omega_m) $, $ (\omega_m,\omega_M) $ and $ (\omega_M,m) $. Then
$ |Cr_\sigma| = 3 $; if $ \sigma\neq\sigma_2 $ only one is a minimum, while
if $ \sigma = \sigma_2 $, there are two minima, by Theorem~\ref{thm.uniqueness}.
\end{proof}
In the formulation of the initial value problem in \eqref{eq.IVP}, we restricted to positive
solutions. However, if $ (R_\omega,\omega) $ is a critical point, then $ (-R_\omega,\omega) $
is also a critical point. So, $ 2|Cr_{\sigma}| $ and $ 2|K_\sigma| $ are the number of critical points and minima of
$ \Er $ on $ M_{\sigma,r}^* $ regardless of the sign.
\begin{remark}
In \cite[Lemma~3.6]{Bon10}, C.~Bonanno proved that for the non-linear Klein-Gordon equation in dimension $ n\geq 3 $ there exist at least one local minimum on $ M_{\sigma,r}^* $ for every connected component of the set $ \{\G < 0\} $, provided $ \sigma $ is large enough. 
The conclusions of \eqref{thm.multiplicity-result.2.4} of Theorem~\ref{thm.multiplicity-result},
can be regarded to as an improvement of this result (even if the dimension is different)
at least in the case \eqref{eq.DP}: there are two local minima 
(in fact two minima) despite the set $ \{\G < 0 \} $ being connected.
\end{remark}
\section{The double power case: non-degeneracy of minima}
\label{sect.dd}
For each level of charge $ \sigma $, we will establish which minima are
degenerate or not. A critical point $ (R_\omega,\omega) $ is degenerate if
and only if the null space of the Hessian
bilinear form restricted to the tangent space
\begin{equation}
\label{eq.88}
T_{(R_\omega,\omega)} M_{\sigma,r}^* = \{(v,\eta)\in H^1 _r(\R;\R)\times\R\mid 
\eta\|R_\omega\|_{L^2}^2 + 2\omega(R_\omega,v)_{L^2} = 0\}.
\end{equation}
is non-trivial. The space above is the kernel of the differential of $ \Chr $ at
the point $ (R_\omega,\omega) $. From Theorem~\ref{thm.minima-structure}, the Lagrange multiplier
of the critical point $ (R_\omega,\omega) $ is $ \omega $. In \cite[Appendix]{GG17} we showed that if
$ \G $ is as in \eqref{eq.DP} then $ \Er $ and $ \Chr $ are $ C^2 (H^1 (\R;\R);\R) $.
The Hessian of $ \Ers $,
as a constrained functional is given by
\[
H(R_\omega,\omega) := D^2 \Ers(R_\omega,\omega) - \omega D^2 \Chrr(R_\omega,\omega).
\]
Given two vectors $ Y := (v,\eta) $ and $ Z := (w,\kappa) $ in $ T_{(R_\omega,\omega)} M_{\sigma,r}^* $,
the Hessians of the two functionals $ \Ers $ and $ \Chrr $ are
\begin{gather*}
\begin{split}
D^2 \Ers(R,\omega)[Y,Z] &= 
\eta\kappa\|R_\omega\|_{L^2}^2 + 2\omega\kappa(R_\omega,v)_{L^2} +  
2\omega\eta(R_\omega,w)_{L^2} \\
&+ D^2_{uu} \Ers(R_\omega,\omega)[v,w]
\end{split}\\
D^2 \Chrr(R_\omega,\omega)[Y,Z] = 2\kappa(R_\omega,v)_{L^2} + 2\eta(R_\omega,w)_{L^2} + 2\omega(v,w)_{L^2}.
\end{gather*}
By inspection, 
\[
\begin{split}
D^2 _{uu} \Ers(R_\omega,\omega)[v,w] &= \nint v'(x)w'(x)dx \\
&+ \nint (G''(R_\omega(x)) + m^2 + \omega^2)v(x)w(x) dx.
\end{split}
\]
Then, 
\begin{equation}
\label{eq.94}
\begin{split}
H(R_\omega,\omega)[Y,Z] &= \nint v'(x)w'(x)dx \\
&+ \nint (G''(R_\omega(x)) + m^2 - \omega^2)v(x)w(x) dx
+ \eta\kappa\|R_\omega\|_{L^2}^2.
\end{split}
\end{equation}
If $ v $ is $ H^2 _r $, then
\begin{equation}
\label{eq.93}
H(R_\omega,\omega)[Y,Z] = (L_+(v),w)_{L^2} + \eta\kappa\|R_\omega\|_{L^2}^2,
\end{equation}
where 
\[
L_+ (v) := -v'' + G''(R_\omega(x)) v + (m^2 - \omega^2)v.
\]
For every $ Y = (v,\eta) $ in $ T_{(R_\omega,\omega)} M_{\sigma,r}^* $, we define
\begin{equation}
\label{eq.123}
\begin{split}
\xi(v,\eta) := H(R_\omega,\omega)[Y,Y] &= 
\nint |v'(x)|^2 dx \\
&+ \nint (G''(R_\omega(x)) + m^2 - \omega^2)v(x)^2 dx
+ \eta^2\|R_\omega\|_{L^2}^2.
\end{split}
\end{equation}
Since $ (R_\omega,\omega) $ is a minimum, the Hessian is a positive semidefinite 
bilinear form. Therefore,
it is non-degenerate if and only if $ H(R_\omega,\omega)[Y,Y] = 0 $ implies that $ Y = 0 $ for every
vector $ Y $. We will follow a similar approach to the one illustrated in the proof of \cite[Proposition~2.9]{Wei85}. We denote with $ S(T_{(R_\omega,\omega)} M_{\sigma,r}^*) $ the unit
sphere with respect to the norm $ \|(v,\eta)\|^2 = \eta^2 + \|v\|_{L^2}^2 $.
\begin{proposition}
The minima $ (R_\omega,\omega) $ and $ (-R_\omega,\omega) $ are non-degenerate if and only if $ \sigma'(\omega)\neq 0 $.
\end{proposition}
\begin{proof}
Firstly, we show that from $ \sigma'(\omega)\neq 0 $ it follows that $ \inf(\xi) > 0 $
on the unit sphere 
$ S(T_{(R_\omega,\omega)} M_{\sigma,r}^*) $. On the contrary, we can suppose that the $ \inf(\xi) = 0 $, because 
$ (R_\omega,\omega) $ is a minimum. We prove that the infimum is achieved. Let $ (\eta_n,v_n) $ be a minimizing sequence in $ S(T_{(R_\omega,\omega)} M_{\sigma,r}^*) $. Since 
$ \eta_n^2 + \|v_n\|_{L^2}^2 = 1 $,
the sequence $ (\eta_n) $ is bounded. Up to extract a subsequence, we can suppose that there
exist $ \eta $ in $ \R $ such that 
\[
\lim_{n\to+\infty} \eta_n = \eta.
\]
Since $ R_\omega $ is in $ L^{\infty} (\R;\R) $, the function $ G'' $ can
be estimated by a single power, that is, $ |G''(R_\omega(x))|\leq c|R_\omega(x)|^2 $ for some 
$ c $ in $ \R $. By applying the Cauchy-Schwarz inequality to $ (|R_\omega|^2,v_n^2)_{L^2} $, and\eqref{eq.122} with $ p = 4 $, we obtain
\[
\begin{split}
\nint G''(R_\omega(x)) v_n(x)^2dx &
\geq -s_{GN}^2 c \|R_\omega\|_{L^4}^{2} \|v_n'\|_{L^2}^{\frac{1}{2}}
\|v_n\|_{L^2}^{\frac{3}{2}}\\
&\geq - s_{GN}^2 c \|R_\omega\|_{L^4}^{2} \|v_n'\|_{L^2}^{\frac{1}{2}}.
\end{split}
\]
The second inequality follows from $ \eta_n ^2 + \|v_n\|_{L^2}^2\leq 1 $.
Since $ \frac{1}{2} < 2 $, the sequence $ (v_n') $ is bounded as well. Then, up to extract
a subsequence, we can suppose that there exist $ v $ such that $ (v_n) $ converges to $ v $ in
$ L^\infty (\R;\R) $ and weakly in $ H^1 (\R;\R) $. 
Then $ \lim_{n\to+\infty} G''(R_\omega)v_n^2 = G''(R_\omega)v^2 $
in $ L^1(\R;\R) $. In fact, the sequence $ (G''(R_\omega)v_n^2) $ converges pointwise almost everywhere to 
$ G''(R_\omega)v^2 $ and it is dominated by $ c|R_\omega|^2(\sup_n \|v_n\|_{\infty}^2) $ which is $ L^1 $, because it
has exponential decay from \eqref{prop.4.3} of Proposition~\ref{prop.4}. From the weak convergence
$ v_n\rightharpoonup v $ in $ L^2 $ it also follows that $ (v,\eta) $ is in $ T_{(R_\omega,\omega)} M_{\sigma,r}^* $. Also $ v\neq 0 $; on the contrary $ (G''(R_\omega)v_n^2) $ would converge to
zero in $ L^1 $ and, from \eqref{eq.123}, we would obtain that
\[
\begin{split}
\xi(v_n,\eta_n)&\geq \nint (G''(R_\omega(x)) + m^2 - \omega^2)v_n(x)^2 dx
+ \eta_n^2\|R_\omega\|_{L^2}^2 \\
&= o(1) + (m^2 - \omega^2)(1 - \eta_n^2) + \eta_n^2\|R_\omega\|_{L^2}^2 \\
&= o(1) + (m^2 - \omega^2)(1 - \eta^2) + \eta^2 \|R_\omega\|_{L^2}^2 
\end{split}
\]
From Theorem~\ref{thm.minima-structure}, $ m^2 - \omega^2 > 0 $. Therefore, taking the
limit, we would obtain that $ \eta^2 $ is equal to 0 and 1 at the same time. Then, 
we can define
\begin{equation}
\label{eq.86}
(v_*,\eta_*) := \left(\frac{v}{\|v\|_{L^2}},\frac{\eta}{\|v\|_{L^2}}\right)
\in S(T_{(R_\omega,\omega)} M_{\sigma,r}^*).
\end{equation}
Summing up, we have
\[
\begin{split}
\xi(\eta_n,v_n) \geq o(1) + \xi(\eta,v) = o(1) + \|v\|_{L^2}^2\,\xi(\eta_*,v_*).
\end{split}
\]
The first inequality is given by the lower-semicontinuity property of the $ L^2 $ norm.
The limit in the first term is zero, while the limit in the third term in non-negative. Therefore,
$ \xi(v_*,\eta_*) = 0 $. Then the infimum of $ \xi $ is achieved. 
Since $ (v_*,\eta_*) $ is a critical point of $ \xi $ constrained to
$ S(T_{(R_\omega,\omega)} M_{\sigma,r}^*) $, there exist
$ \beta,\gamma > 0 $ such that 
\[
\nabla\xi(v_*,\eta_*) = \beta\big(2\omega R_\omega,\|R_\omega\|_{L^2}^2\big) + \gamma(v_*,\eta_*)
\]
Since $ \nabla\xi(v_*,\eta_*)\cdot(v_*,\eta_*) = 2\xi(v_*,\eta_*) = 0 $, we have $ \gamma = 0 $.
This means that $ (v_*,\eta_*) $ is a solution of the system
\[
-2v_* '' + 2(G''(R_\omega(x)) + m^2 - \omega^2)v_* = 2\omega\beta R_\omega,\quad 
2\eta_*\|R_\omega\|_{L^2}^2 = \beta\|R_\omega\|_{L^2}^2
\]
which gives
\begin{equation*}
v_* '' - (G''(R_\omega(x)) + m^2 - \omega^2)v_* = -2\omega\eta_* R_\omega.
\end{equation*}
Taking the derivative with respect to $ \omega $ in \eqref{eq.IVP}, we obtain
\begin{equation}
\label{eq.96}
(\partial_\omega R_\omega)'' - (G''(R_\omega(x))
 + m^2 - \omega^2) \partial_\omega R_\omega = - 2\omega R_\omega.
\end{equation}
Therefore, $ w := v_* - \eta_*\partial_\omega R_\omega $, is a solution to the differential
equation
\begin{equation}
\label{eq.95}
w'' - (G''(R_\omega(x)) + m^2  - \omega^2)w = 0.
\end{equation}
Since both $ v_* $ and $ \partial_\omega R_\omega $ are even, $ w $ is also even. However,
the space of solutions in $ H^1 (\R;\R) $ of the equation above is generated by $ R'_\omega $, which is an odd function. Therefore, $ w\equiv 0 $ and $ v_* = \eta_*\partial_\omega R_\omega $.
Since $ (v_*,\eta_*) $ is orthogonal to $ (2\omega R_\omega,\|R_\omega\|_{L^2}^2) $, we have
\[
\begin{split}
0 &= (v_*,\eta_*)\cdot (2\omega R_\omega,\|R_\omega\|_{L^2}^2) = 
(\eta_*\partial_\omega R_\omega,\eta_*)\cdot (2\omega R_\omega,\|R_\omega\|_{L^2}^2) \\
&= \eta_* (\partial_\omega R_\omega,1)\cdot (2\omega R_\omega,\|R_\omega\|_{L^2}^2) = \eta_*\sigma'(\omega).
\end{split}
\]
We can rule out the case $ \eta_*  = 0 $ and then obtain a contradiction with the assumption
that $ \sigma'(\omega)\neq 0 $.
If $ \eta_* = 0 $, then $ v_* $ is an even solution to \eqref{eq.95}
and thus equal to 0, which contradicts \eqref{eq.86}. This proves one of the two
implications of the lemma.

Conversely, if $ (R_\omega,\omega) $ is non-degenerate, that is the infimum of $ \xi $
on $ S(T_{(R_\omega,\omega)} M_{\sigma,r}^*) $ is positive, then $ \sigma'(\omega)\neq 0 $. 
On the contrary
\[
0 = \sigma'(\omega) = \|R_\omega\|_{L^2}^2 + 2\omega (R_\omega,\partial_\omega R_\omega)_{L^2} = 
(\partial_\omega R_\omega,1)\cdot
(2\omega R_\omega,\|R_\omega\|_{L^2}^2).
\]
Therefore, $ (\partial_\omega R_\omega,1) $ belongs to $ T_{(R_\omega,\omega)} M_{\sigma,r}^* $. 
Taking the scalar product in $ L^2 (\R;\R) $ with 
$ -\partial_\omega R_\omega $ in \eqref{eq.96}, we obtain
\[
\begin{split}
0 &= \nint |(\partial_\omega R_\omega)'|^{2}dx
+\nint (G''(R_\omega(x))
 + m^2 - \omega^2) (\partial_\omega R_\omega)^2dx \\
&- 2\nint\omega R_\omega\partial_\omega R_\omega dx\\
&=\nint |(\partial_\omega R_\omega)'|^{2}dx
+\nint (G''(R_\omega(x))
 + m^2 - \omega^2) (\partial_\omega R_\omega)^2dx
+ \|R_\omega\|_{L^2}^2 \\
&= \xi(\partial_\omega R_\omega,1).
\end{split}
\]
The second equality follows from $ \sigma'(\omega) = 0 $. Therefore, if we set 
\[
T := \frac{(\partial_\omega R_\omega,1)}{\sqrt{1 + \|\partial_\omega R_\omega\|_{L^2}^2}}
\]
we have $ \xi(T)  = 0 $ contradicting the non-degeneracy assumption.
\end{proof}
\begin{theorem}
\label{thm.degeneracy}
For every $ a,b,m $ such that $ \tau(a,b,m) > 1 $, 
\begin{enumerate}[(i)]
\item if $ \tau\neq\tau_* $, then minima are non-degenerate
\item if $ \tau = \tau_* $ and $ \sigma\neq\sigma_s $, then minima are non-degenerate.
If $ \sigma = \sigma_s $, then $ (R_{\omega_s},\omega_s) $ and
$ (-R_{\omega_s},\omega_s) $ are degenerate.
\end{enumerate}
\end{theorem}
\begin{proof}
(i). If $ \tau > \tau_* $, $ \sigma' $ is always negative. Therefore, minima are non-degenerate
from Theorem~\ref{thm.degeneracy}. If $ 1 < \tau < \tau_* $, the derivative of $ \sigma $ 
vanishes at $ \omega_M $ and $ \omega_m $; $ \sigma_M $ is achieved at $ \omega_M $ and
another point $ \omega_1 < \omega_M $. We repeat the computation in \eqref{eq.99} with
$ \omega_M $ replacing $ \omega_3 $ and $ \sigma_M $ in place of $ \sigma $. Therefore,
\[
e(\omega_M) - e(\omega_1) = \int_{\omega_1}^{\omega_M} (\sigma_M - \sigma(t))dt > 0
\]
showing that $ \Er(R_{\omega_M},\omega_M) > \I(\sigma_M) $ and proving that minima
on $ M_{\sigma_M,r}^* $ are non-degenerate. Similarly, $ \sigma_m $ is achieved at
$ \omega_m $ and another point $ \omega_m < \omega_3 $. We repeat the computation in \eqref{eq.99} with $ \omega_m $ replacing 
$ \omega_1 $ and $ \sigma_m $ in place of $ \sigma_2 $. Therefore,
\[
e(\omega_m) - e(\omega_1) = \int_{\omega_1}^{\omega_m} (\sigma_m - \sigma(t))dt > 0
\]
proving that $ \Ers(R_{\omega_m},\omega_m) > \I(\sigma_m) $. Therefore, minima on
$ M_{\sigma_m,r}^* $ are non-degenerate. 

(ii). From Lemma~\ref{lem.decreasing-charge}, $ \sigma'(\omega)\neq 0 $ unless $ \omega = \omega_s $; from Theorem~\ref{thm.uniqueness}, on $ M_{\sigma_s,r}^* $ the only critical points are minima.
Therefore, both $ (R_{\omega_s},\omega_s) $ and $ (-R_{\omega_s},\omega_s) $ are degenerate
minima.
\end{proof}
\section{The double power case: stability of standing-waves}
\label{sect.sw}
In this conclusive section, we prove the orbital stability of standing-wave solutions to
\eqref{eq.NLKG} when $ \G $ is the double power $ \eqref{eq.DP} $.
\begin{theorem}
\label{thm.stability-wave}
For the non-linear term \eqref{eq.DP},
the set $ \Gamma(\Phi) $ is stable for every $ \sigma > 0 $ and $ \Phi $ in $ \Gamma_\sigma $.
\end{theorem}
\begin{proof}
We start by looking at the cases \eqref{thm.multiplicity-result.1}, 
\eqref{thm.multiplicity-result.2.1}, \eqref{thm.multiplicity-result.2.2}, 
\eqref{thm.multiplicity-result.2.3} of Theorem~\ref{thm.multiplicity-result}, where there are exactly two minima, $ (R_\omega,\omega) $
and $ (-R_\omega,\omega) $. We claim that 
\[
\Gamma_\sigma = \Gamma(R_\omega,-i\omega R_\omega).
\]
The set on the right is a subset of the ground-state by the invariance properties described
in \eqref{eq.77}.
From Theorem~\ref{thm.minima-structure}, 
given $ \Phi $ in $ \Gamma_\sigma $, there exist a pair $ (R_1,\omega_1) $ and $ (y,z) $ in 
$ \R\times S^1 $ such that 
\[
\Phi = (zR_1(\cdot + y),-i\omega_1 z R_1(\cdot + y)),\quad \Ers(R_1,\omega_1) = \I(\sigma).
\]
By Theorem~\ref{thm.multiplicity-result}, 
$ R_1 = R_\omega $ implying that $ \Phi $ is an element of 
$ \Gamma_\sigma(R_\omega,-i\omega R_\omega) $. Therefore, since $ \Gamma_\sigma $ is stable by
Theorem~\ref{thm.ground-state-stability}, the set $ \Gw(\Phi) $ is also stable. 
In the case \eqref{thm.multiplicity-result.2.4} of
Theorem~\ref{thm.multiplicity-result} there are two positive minima $ (\omega_1,R_{\omega_1}) $ 
and $ (\omega_2,R_{\omega_2}) $ in $ \M_{\sigma_2,r}^* $. Firstly, we show that 
\begin{equation}
\label{eq.17}
\delta := \mathrm{dist}\big(\Gamma(R_{\omega_1},-i\omega_1 R_{\omega_1}),
\Gamma(R_{\omega_2},-i\omega_2 R_{\omega_2})\big) > 0.
\end{equation}
We consider two arbitrary points $ \Phi_1 $ and $ \Phi_2 $ of the two sets. Therefore, there
are $ (y_1,z_1) $ and $ (y_2,z_2) $ in $ \R\times S^1 $ such that 
\begin{gather*}
\Phi_1 = (z_1R_{\omega_1}(\cdot + y_1),-iz_1\omega_1 R_{\omega_1}(\cdot + y_1)),\ 
\Phi_2 = (z_2R_{\omega_2}(\cdot + y_2),-iz_2\omega_2 R_{\omega_2}(\cdot + y_2))
\end{gather*}
according to the definition given in \eqref{eq.ssw}. We have
\[
\begin{split}
\dist(\Phi_1,\Phi_2)^2 &= \|z_1R_{\omega_1}(\cdot + y_1) - z_2R_{\omega_2}(\cdot + y_2)\|_{H^1}^2\\
&+ \|iz_1\omega_1R_{\omega_1}(\cdot + y_1)-iz_2\omega_2R_{\omega_2}(\cdot + y_2)\|_{L^2}^2.
\end{split}
\]
From \eqref{prop.4.2} of Proposition~\ref{prop.4}, $ R_{\omega_1} $ and $ R_{\omega_2} $ are positive and radially decreasing. Therefore
\[
\begin{split}
&\|z_1R_{\omega_1}(\cdot + y_1) - z_2R_{\omega_2}(\cdot + y_2)\|_{H^1}^2\geq
\|z_1R_{\omega_1}(\cdot + y_1) - z_2R_{\omega_2}(\cdot + y_2)\|_{L^2}^2\\
=&\|R_{\omega_1} - z_2\overline{z}_1 R_{\omega_2}(\cdot + y_2 - y_1)\|_{L^2}^2
\geq\|R_{\omega_1} - R_{\omega_2}\|_{L^2}^2.
\end{split}
\]
The equality follows from a variable change. Therefore,
\begin{equation*}
\dist(\Phi_1,\Phi_2)\geq\|R_{\omega_1} - R_{\omega_2}\|_{L^2}
\end{equation*}
which provides a lower bound for $ \delta $. Hereafter, we will use the notation
\begin{equation}
\label{eq.124}
S_1 := \Gamma(R_{\omega_1},-i\omega_1 R_{\omega_1}),\quad
S_2 := \Gamma(R_{\omega_2},-i\omega_2 R_{\omega_2}).
\end{equation}
Since $ \delta = \dist(S_1,S_2) > 0 $, in the metric space $ \X $ each of the sets
$ S_i $ are isolated from each other. In fact,
\begin{equation}
\label{eq.isolated}
B(S_i,\delta)\cap\Gamma_{\sigma_2} = S_i.
\end{equation}
For $ 1\leq i\leq 2 $. We define $ \E_i := \inf\{\E(\Phi)\mid \Phi\in\partial B(S_i,\delta/2)\cap\M_{\sigma_2}\} $. We claim that
\begin{equation}
\label{eq.111}
\E_i > I(\sigma_2).
\end{equation}
Otherwise, we would have a sequence $ (\Phi_n) $ such that
\begin{equation}
\label{eq.125}
\E(\Phi_n)\to I(\sigma),\quad \Ch(\Phi_n) = \sigma_2.
\end{equation}
By Lemma~\ref{lem.concentration-compactness}, up to extract a subsequence, we can suppose that there exist $ \Phi $ in $ \Gamma_{\sigma_2} $ and $ (y_n)\subseteq\R $ such that
\begin{equation*}
\Phi_{n} (\cdot + y_n)\to\Phi\text{ in } H^1(\R;\C)\times L^2(\R;\C).
\end{equation*}
Therefore, $ \Phi $ is in $ \partial B(S_i,\delta/2)\cap\Gamma_{\sigma_2} $, because 
$ \dist(\Phi_{n}(\cdot + y_n),S_i) = \dist(\Phi_{n},S_i) $. Then, we obtained a 
contradiction with \eqref{eq.isolated}. We are now able to prove that each of these
sets is stable. Suppose that $ S_i $ is not stable for some $ i $ in $ \{1,2\} $; the other
set is $ S_{3 - i} $.
Then, there are sequences $ (\Phi_n) $, 
$ (t_n) $ and $ \varepsilon_0 > 0 $ such that
\begin{equation}
\label{eq.100}
\mathrm{dist}(\Phi_n,S_i)\to 0,\quad
\mathrm{dist}(U(t_n,\Phi_n),S_i)\geq\varepsilon_0.
\end{equation}
We set $ \Psi_n := U(t_n,\Phi_n) $. From Theorem~\ref{thm.ground-state-stability}, 
the set $ \Gamma_{\sigma_2} $ is stable, which implies
$ \dist(\Psi_n,\Gamma_{\sigma_2})\to 0 $. From \eqref{eq.17}, \eqref{eq.124} 
and \eqref{eq.100}, 
$ \dist(\Psi_n,S_{3 - i})\to 0 $. Then, there exist $ n_0 $ such that for every $ n\geq n_0 $
$ \dist(\Psi_n,S_{3 - i}) < \frac{\delta}{2} $. Then, by \eqref{eq.isolated} 
$ \dist(\Psi_n,S_i)\geq\delta/2 $. From \eqref{eq.101}, the following curves
\[
\alpha_n\colon (-\infty,+\infty)\to\X,\quad
\alpha_n (t) = \left(\frac{\sigma_2}{\Ch(\Phi_n)}\right)^{\frac{1}{2}} U(t,\Phi_n)
\]
all belong to $ \M_{\sigma_2} $. From \eqref{eq.125}, 
we can also suppose that $ \E(\alpha_n (0)) < E_i $ for every $ n\geq n_0 $. Therefore, 
\begin{gather}
\label{eq.15}
\E(\alpha_{n_0} (0)) < \E_i\\
\label{eq.16}
\dist(\alpha_{n_0}(0),S_i) < \frac{\delta}{2},\quad
\dist(\alpha_{n_0}(t_{n_0}),S_i)\geq\frac{\delta}{2}.
\end{gather}
From \eqref{eq.16}, there exist $ \overline{t} $ such that 
$ \dist(\alpha_{n_0} (\overline{t}),S_i) = \delta/2 $. Then $ \alpha_{n_0} (\overline{t}) $
is in $ \partial B(S_i,\delta/2)\cap M_{\sigma_2} $. From \eqref{eq.101},
$ \E(\alpha_{n_0}(0)) = \E(\alpha_{n_0} (\overline{t})) $. Therefore, 
$ \E(\alpha_{n_0}(\overline{t})) < \E_i $, which contradicts \eqref{eq.111}.
\end{proof}
\subsection*{The case $ \sigma = 0 $}
We conclude this section by looking at the case $ \sigma = 0 $, which we include for the
sake of completeness. If $ \sigma = 0 $, from \eqref{R1} the minimum of $ \E $ is 
achieved only by the pair $ (0,0)\in \X $. Therefore, $ \Gamma_0 = \{(0,0)\} $. The
stability of $ \Gamma_0 $ relies only on the continuity and the coercivity of $ \E $.
\begin{proposition}
$ \Gamma_0 $ is stable.
\end{proposition}
\begin{proof}
Given $ \var > 0 $, by \eqref{prop.1.3} of Proposition~\ref{prop.1} there exist $ e > 0 $ 
such that $ h_1 (e) < \var $. By \eqref{prop.1.5} of the same proposition, there exist
$ \delta > 0 $ such that $ \|\Phi\|_\X < \delta $ implies $ \E(\Phi) < e $. We also
notice that $ \dist(\Phi,\Gamma_0) = \|\Phi\|_\X $. Then, given $ \Phi $ such that 
$ \dist(\Phi,\Gamma_0) < \delta $, we also have $ \|\Phi\|_\X < \delta $.
Then $ \E(\Phi) < e $ by the continuity of $ \E $. From \eqref{eq.101}, 
$ \E(U(t,\Phi)) < e $ for every $ t $. Then $ \|U(t,\Phi)\|_\X < h_2 (e) < \var $
by \eqref{prop.1.3} of Proposition~\ref{prop.1}. Therefore $ \dist(U(t,\Phi)) < \var $ for every $ t $, which concludes the proof.
\end{proof}
\def\cprime{$'$} \def\cprime{$'$} \def\cprime{$'$} \def\cprime{$'$}
  \def\cprime{$'$} \def\cprime{$'$} \def\cprime{$'$} \def\cprime{$'$}
  \def\cprime{$'$} \def\polhk#1{\setbox0=\hbox{#1}{\ooalign{\hidewidth
  \lower1.5ex\hbox{`}\hidewidth\crcr\unhbox0}}} \def\cprime{$'$}
  \def\cprime{$'$} \def\cprime{$'$}

\end{document}